\documentclass{amsart}

\usepackage{amsfonts}
\usepackage{amssymb}
\usepackage{amsmath}
\usepackage{hyperref}
\usepackage{mathrsfs}
\usepackage{centernot}
\usepackage{mathdots}
\usepackage{stmaryrd}
\usepackage[all]{xy}

\newtheorem{thm}{Theorem}[section]
\newtheorem{lem}[thm]{Lemma}
\newtheorem{prp}[thm]{Proposition}
\newtheorem{cor}[thm]{Corollary}
\newtheorem{dfn}[thm]{Definition}

\newtheorem{baseexample}[thm]{Example} 
\newtheorem{baseremark}[thm]{Remark} 

\newenvironment{example}
{\begin{baseexample}\rm}{\end{baseexample}}
\newenvironment{remark}
{\begin{baseremark}\rm}{\end{baseremark}}

\newcommand{\rem}[1]{}

\newcommand{\Step}[1]{\noindent {\bf Step #1.}} 

\newcommand{\F}{\mathbb{F}}

\newcommand{\N}{\mathbb{N}}

\newcommand{\Z}{\mathbb{Z}}

\newcommand{\calCapital}{
\newcommand{\calA}{{\mathcal{A}}}
\newcommand{\calB}{{\mathcal{B}}}
\newcommand{\calC}{{\mathcal{C}}}
\newcommand{\calD}{{\mathcal{D}}}
\newcommand{\calE}{{\mathcal{E}}}
\newcommand{\calF}{{\mathcal{F}}}
\newcommand{\calG}{{\mathcal{G}}}
\newcommand{\calH}{{\mathcal{H}}}
\newcommand{\calI}{{\mathcal{I}}}
\newcommand{\calJ}{{\mathcal{J}}}
\newcommand{\calK}{{\mathcal{K}}}
\newcommand{\calL}{{\mathcal{L}}}
\newcommand{\calM}{{\mathcal{M}}}
\newcommand{\calN}{{\mathcal{N}}}
\newcommand{\calO}{{\mathcal{O}}}
\newcommand{\calP}{{\mathcal{P}}}
\newcommand{\calQ}{{\mathcal{Q}}}
\newcommand{\calR}{{\mathcal{R}}}
\newcommand{\calS}{{\mathcal{S}}}
\newcommand{\calT}{{\mathcal{T}}}
\newcommand{\calU}{{\mathcal{U}}}
\newcommand{\calV}{{\mathcal{V}}}
\newcommand{\calW}{{\mathcal{W}}}
\newcommand{\calX}{{\mathcal{X}}}
\newcommand{\calY}{{\mathcal{Y}}}
\newcommand{\calZ}{{\mathcal{Z}}}
}

\newcommand{\bbCapital}{
\newcommand{\bbA}{{\mathbb{A}}}
\newcommand{\bbB}{{\mathbb{B}}}
\newcommand{\bbC}{{\mathbb{C}}}
\newcommand{\bbD}{{\mathbb{D}}}
\newcommand{\bbE}{{\mathbb{E}}}
\newcommand{\bbF}{{\mathbb{F}}}
\newcommand{\bbG}{{\mathbb{G}}}
\newcommand{\bbH}{{\mathbb{H}}}
\newcommand{\bbI}{{\mathbb{I}}}
\newcommand{\bbJ}{{\mathbb{J}}}
\newcommand{\bbK}{{\mathbb{K}}}
\newcommand{\bbL}{{\mathbb{L}}}
\newcommand{\bbM}{{\mathbb{M}}}
\newcommand{\bbN}{{\mathbb{N}}}
\newcommand{\bbO}{{\mathbb{O}}}
\newcommand{\bbP}{{\mathbb{P}}}
\newcommand{\bbQ}{{\mathbb{Q}}}
\newcommand{\bbR}{{\mathbb{R}}}
\newcommand{\bbS}{{\mathbb{S}}}
\newcommand{\bbT}{{\mathbb{T}}}
\newcommand{\bbU}{{\mathbb{U}}}
\newcommand{\bbV}{{\mathbb{V}}}
\newcommand{\bbW}{{\mathbb{W}}}
\newcommand{\bbX}{{\mathbb{X}}}
\newcommand{\bbY}{{\mathbb{Y}}}
\newcommand{\bbZ}{{\mathbb{Z}}}
}

\newcommand{\catCapital}{
\newcommand{\catA}{{\mathscr{A}}}
\newcommand{\catB}{{\mathscr{B}}}
\newcommand{\catC}{{\mathscr{C}}}
\newcommand{\catD}{{\mathscr{D}}}
\newcommand{\catE}{{\mathscr{E}}}
\newcommand{\catF}{{\mathscr{F}}}
\newcommand{\catG}{{\mathscr{G}}}
\newcommand{\catH}{{\mathscr{H}}}
\newcommand{\catI}{{\mathscr{I}}}
\newcommand{\catJ}{{\mathscr{J}}}
\newcommand{\catK}{{\mathscr{K}}}
\newcommand{\catL}{{\mathscr{L}}}
\newcommand{\catM}{{\mathscr{M}}}
\newcommand{\catN}{{\mathscr{N}}}
\newcommand{\catO}{{\mathscr{O}}}
\newcommand{\catP}{{\mathscr{P}}}
\newcommand{\catQ}{{\mathscr{Q}}}
\newcommand{\catR}{{\mathscr{R}}}
\newcommand{\catS}{{\mathscr{S}}}
\newcommand{\catT}{{\mathscr{T}}}
\newcommand{\catU}{{\mathscr{U}}}
\newcommand{\catV}{{\mathscr{V}}}
\newcommand{\catW}{{\mathscr{W}}}
\newcommand{\catX}{{\mathscr{X}}}
\newcommand{\catY}{{\mathscr{Y}}}
\newcommand{\catZ}{{\mathscr{Z}}}
}

\newcommand{\veps}{\varepsilon}
\newcommand{\vphi}{\varphi}

\newcommand{\derives}{\Longrightarrow}

\newcommand{\suchthat}{\,:\,}
\newcommand{\where}{\,|\,}

\newcommand{\quo}[1]{\overline{#1}}

\DeclareMathOperator{\Cent}{Cent} %
\DeclareMathOperator{\Char}{char} %
\DeclareMathOperator{\End}{End} %
\DeclareMathOperator{\Hom}{Hom} %
\DeclareMathOperator{\id}{id} %
\DeclareMathOperator{\im}{im} %
\DeclareMathOperator{\ind}{ind} %
\DeclareMathOperator{\Jac}{Jac} %

\DeclareMathOperator{\Mor}{Mor} %
\newcommand{\op}{\mathrm{op}} %
\newcommand{\nGL}[2]{\mathrm{GL}_{#2}({#1})}

\newcommand{\nMat}[2]{\mathrm{M}_{#2}(#1)}

\newcommand{\invlim}{\underleftarrow{\lim}\,}

\newcommand{\units}[1]{{#1^\times}}

\newcommand{\rproj}[1]{{\mathrm{proj}}\textrm{-}{#1}}

\calCapital %
\bbCapital %
\catCapital %

\newcommand{\Quad}[2][]{\mathrm{UQ}^{#1}(#2)}     

\newcommand{\TDA}[2][]{{\mathrm{A\widetilde{r}}}_{2#1}({#2})}       

\renewcommand{\rproj}[1]{\catP(#1)}               

\newcommand{\lAd}[1]{L_{#1}}                      

\newcommand{\rdnorm}[1][]{\mathrm{Nrd}_{#1}}      

\newcommand{\Sesq}[1]{\calS_{#1}}                 

\newcommand{\efinv}[3]{({}_{#1}{#3}_{#2})^{\circ}}

\numberwithin{equation}{section}

\renewcommand{\Step}[1]{\noindent {\it Step #1.}} 

\title[Witt's Extension Theorem]{Witt's Extension Theorem for Quadratic Spaces over Semiperfect Rings}

\author{Uriya A.\ First}
\address{Hebrew University of Jerusalem, Israel.}

\date{\today}

\thanks{
The  author was  supported by  an ERC grant \#226135 and an SNFS grant \#IZK0Z2\_151061.
}


\begin{document}

\maketitle

\begin{abstract}
    We prove that every  isometry of between (not-ne\-ces\-sari\-ly orthogonal) summands
    of a unimodular quadratic space over a \emph{semiperfect} ring can be extended an isometry of the whole quadratic space.
    The same result was proved by Reiter for the broader class of \emph{semilocal} rings, but with certain restrictions on
    the base modules, which cannot
    be removed in general.

    Our result implies that unimodular quadratic spaces over semiperfect rings cancel from orthogonal sums.
    This improves a  cancellation result  of Quebbemann, Scharlau and Schulte, which applies
    to quadratic spaces over \emph{hermitian categories}.
    Combining this with other known results yields further  cancellation theorems. For instance,
    we prove cancellation of (1) \emph{systems} of sesquilinear forms over
    henselian local rings, and (2) \emph{non-unimodular} hermitian forms over (arbitrary) valuation rings.

    Finally, we determine the group generated by the reflections
    of a unimodular quadratic space over a semiperfect ring.
\end{abstract}

\section{Introduction}

    Let $F$ be a field of characteristic not $2$ and let $(V,q)$ be a nondegenerate quadratic space
    over $F$. The following theorem, known as \emph{Witt's Theorem} or \emph{Witt's Extension Theorem}, is
    fundamental in the theory of quadratic forms.

    \begin{thm}[Witt]
        Let $V_1,V_2$ be subspaces of $V$ and let $\psi:q|_{V_1}\to q|_{V_2}$ be an isometry. Then
        $\psi$ extends to an isometry $\psi'$ of $q$. Furthermore, $\psi'$ is a product of reflections.
    \end{thm}

    Among the theorem's consequences are cancellation of nondegenerate quadratic spaces and
    the fact that $O(V,q)$, the isometry group of $(V,q)$, acts transitively on maximal totally isotropic subspaces
    of $V$.

\medskip

    The works of Bak \cite{Bak69}, Wall \cite{Wall70} and others have led to defining a notion
    of quadratic forms over arbitrary (non-commutative) rings, and also to an appropriate definition
    of reflections (see \cite{Kne69}, \cite{Knes70}, \cite{Reiter75}).
    In this context, Witt's Extension Theorem was generalized by Reiter \cite{Reiter75} to semilocal rings, but with certain
    restrictions on the quadratic spaces (see also \cite{Kne69} and \cite{Knes70} for earlier results).
    Cancellation of unimodular quadratic spaces was likewise generalized to various families of semilocal rings $A$ including the cases
    where $A$ is commutative (\cite{Roy68}, \cite{Kne69} or
    \cite[\S3.4.3]{Keller88}), local (\cite[Rm.~3.4.2]{Keller88},
    for instance), or $A=\invlim \{A/\Jac(A)^n\}_{n\in\N}$
    (\cite[\S3.4]{QuSchSch79}; this case includes all one-sided artinian rings).
    However, despite the previous evidence, Keller
    \cite{Keller88} has demonstrated that cancellation fails over arbitrary semilocal rings, implying that
    the restrictions in Reiter's Theorem cannot be removed in general.
    (Most of the results mentioned here can also be found in \cite[Ch.~VI]{Kn91}.)
    
\medskip

    In this paper, we restrict our attention to a family of semilocal rings called \emph{semiperfect} rings, and study to
    what extent Witt's Extension Theorem holds in this setting. Recall
    that a ring $A$ is called \emph{semiperfect} if it is semilocal and its Jacobson radical, $\Jac(A)$,
    is idempotent lifting. For example, local rings and semilocal rings satisfying $A=\invlim \{A/\Jac(A)^n\}_{n\in\N}$
    are semiperfect. See \ref{subsection:semiperfect} below for further examples and details.

    Let $(P,[\beta])$ be a unimodular quadratic space over a semiperfect ring (the definition is recalled
    below) and let $Q,S$ be summands of $P$. Our main results are:
    \begin{enumerate}
        \item[(1)] Every isometry $(Q,[\beta|_Q])\to (S,[\beta|_S])$ extends to an isometry of $(P,[\beta])$
        (Corollary~\ref{CR:main-I}, Theorem~\ref{TH:Witt-II}).
        \item[(2)] Under mild assumptions, every isometry of $(P,[\beta])$ is a product of \emph{quasi-reflections}
        (Corollary~\ref{CR:main-III}, Theorem~\ref{TH:Witt-I}).
        \item[(3)] We determine the subgroup of $O(P,[\beta])$ generated by \emph{reflections} (Theorem~\ref{TH:gen-by-reflections}).
        Apart from an obvious exception
        in which there are no reflections, this subgroup is always of finite index in $O(P,[\beta])$.
    \end{enumerate}
    The proofs are based on Reiter's ideas with certain improvements. In particular, we
    generalize Reiter's \emph{$e$-reflections} (see \cite{Reiter75}).
    We also stress that (1)--(3)  hold without assuming $2$ is invertible.

\medskip

    Our results imply that unimodular quadratic spaces over semiperfect rings cancel from orthogonal sums.
    (Note that the base ring in Keller's counterexample \cite[\S2]{Keller88} is semilocal but not semiperfect.)
    This in turn leads to other cancellation theorems as follows: In \cite{BayerFain96}, \cite{BayerMold12}
    and \cite{BayFiMol13}, it was shown that \emph{systems} of (not-necessarily unimodular) \emph{sesquilinear} forms can be treated as
    (single) unimodular hermitian forms over a different base ring. Thus, cancellation holds
    when this base ring is semiperfect. Using this, we show that cancellation holds for
    \begin{enumerate}
        \item[(a)] arbitrary (i.e.\ not-necessarily unimodular) hermitian forms over involutary valuation rings (Corollary~\ref{CR:non-unimod}),
        \item[(b)] systems of sesquilinear forms over involutary henselian valuation rings (Corollary~\ref{CR:systems-II}).
    \end{enumerate}
    We also strengthen a cancellation theorem of Quebbeman, Scharlau and Schulte \cite[\S3.4]{QuSchSch79}
    which applies to quadratic spaces over \emph{hermitian categories} (Corollary~\ref{CR:cancellation-herm-cats}).
    Specifically, the cancellation of \cite[\S3.4]{QuSchSch79} assumes that the underlying hermitian category satisfies:
    (i) all idempotents split, (ii) every object is the direct sum of objects with local endomorphism ring, and (iii)
    if $A$ is the endomorphism ring of an object, then $A=\invlim \{A/\Jac(A)^n\}_{n\in\N}$.
    We show that cancellation holds even without assuming condition (iii).

\medskip

    The paper is organized as follows:
    In section~\ref{section:preliminaries}, we recall the definitions of quadratic forms over rings and several results to be used
    throughout. Section~\ref{section:reflection} introduces quasi-reflections and reflections. In
    section~\ref{section:Witt-thm}, we prove our version of Witt's Extension Theorem and discuss its applications.
    Finally, in section~\ref{section:gen-by-ref}, we describe the group spanned by the reflections of a unimodular
    quadratic space (over a semiperfect ring).

\section{Preliminaries}
\label{section:preliminaries}

    This section collects several preliminary
    topics that will be used throughout the paper: We recall quadratic forms over unitary rings,  several facts concerning them,
    a notion of orthogonality for unitary rings, and several facts about semiperfect rings.

\subsection{Quadratic Forms}
\label{subsection:unitary-rings}

    We start with recalling quadratic forms. The definitions
    go back to Bak \cite{Bak69} and Wall \cite{Wall70}. See \cite{Bass73AlgebraicKThyIII},
    \cite[Ch.~7]{SchQuadraticAndHermitianForms}, or \cite{Kn91}
    for an extensive discussion.

\medskip

    Let $A$ be a ring. An \emph{anti-structure} on $A$ consists of a pair $(\sigma,u)$ such
    that $\sigma:A\to A$ is an anti-automorphism (written exponentially) and $u\in \units{A}$
    satisfies $u^\sigma u=1$ and $a^{\sigma\sigma}=uau^{-1}$ for all $a\in A$.

    Denote by $\rproj{A}$ the category of finitely generated projective right $A$-modules.
    As usual, a \emph{sesquilinear space} is a pair $(P,\beta)$
    such that $P\in\rproj{A}$ and $\beta:P\times P\to A$ is a biadditive map
    satisfying
    $
    \beta(xa,yb)=a^\sigma\beta(x,y)b
    $ for all  $x,y\in P$, $a,b\in A$.
    In this case, we call $\beta$ a \emph{sesquilinear form}.
    The form $\beta$ is called \emph{$u$-hermitian} if it also satisfies
    $
    \beta(x,y)=\beta(y,x)^\sigma u
    $.

    We say that $(P,\beta)$ is \emph{unimodular} if the map $\lAd{\beta}:P\to P^*:=\Hom_A(P,A)$
    given by sending $x\in P$ to $[y\mapsto \beta(x,y)]\in P^*$ is an isomorphism.
    Note that $P^*$ can be made into a \emph{right} $A$-module by setting
    $
    (f a)x=a^\sigma(fx)
    $
    for all $f\in P^*$, $a\in A$, $x\in P$. This makes $\lAd{\beta}$ is a homomorphism of $A$-modules.
    There is a \emph{natural isomorphism} $\omega_P:P\to P^{**}$ give by
    $
    (\omega_P f)x=(fx)^\sigma u
    $
    for all $f\in P^*$, $x\in P$.

\medskip

    Next, set
    $
    \Lambda^{\min}=\{a-a^\sigma u\where a\in A\}$
    and
    $\Lambda^{\max}= \{a\in A\suchthat a^\sigma u=-a\}
    $.
    A \emph{form parameter}
    (for $(A,\sigma,u)$) consists of an additive group $\Lambda$ such that
    \[\Lambda^{\min}\subseteq\Lambda\subseteq\Lambda^{\max}\qquad \text{and}\qquad
    a^\sigma\Lambda a\subseteq\Lambda\qquad\forall\, a\in A\ .
    \]
    In this case, the quartet $(A,\sigma,u,\Lambda)$ is called a \emph{unitary ring}.

\medskip

    For $P\in\rproj{A}$,  let $\Sesq{P}$ denote the abelian group of sesquilinear
    forms on $P$, and let $\Lambda_P$ denote
    the subgroup consisting of sesquilinear forms $\gamma\in\Sesq{P}$ satisfying
    $\gamma(x,y)=-\gamma(y,x)^\sigma u$ and $\gamma(x,x)\in\Lambda$ for all $x,y\in P$.
    The image of $\beta\in \Sesq{P}$ in $\Sesq{P}/\Lambda_P$ is denoted $[\beta]$.

    A \emph{quadratic space} (over $(A,\sigma,u,\Lambda)$) is a pair $(P,[\beta])$ with $P\in\rproj{A}$
    and $[\beta]\in\Sesq{P}/\Lambda_P$. Associated with $[\beta]$ are the $u$-hermitian form
    \[h_\beta(x,y)=\beta(x,y)+\beta(y,x)^\sigma u\]
    and the
    quadratic map $\hat{\beta}:P\to A/\Lambda$
    given by
    \[\hat{\beta}(x)=\beta(x,x)+\Lambda\ .\]
    Both $h_\beta$ and $\hat{\beta}$ are determined by the class $[\beta]$, and conversely,
    $[\beta]$ is determined by $h_\beta$ and $\hat{\beta}$.
    We also have
    \begin{equation}\label{EQ:beta-h-connection}
    \hat{\beta}(x+y)=\hat{\beta}(x)+\hat{\beta}(y)+h_\beta(x,y)\qquad\forall\, x,y\in P\ .
    \end{equation}
    We say that $(P,[\beta])$ is \emph{unimodular} if $h_\beta$ is a unimodular
    $u$-hermitian form, i.e.\ if $\lAd{[\beta]}:=\lAd{h_\beta}$ is an isomorphism ($\beta$ itself may be non-unimodular).

\medskip

    Isometries between quadratic (resp.\ sesquilinear, $u$-hermitian) spaces
    are defined in the standard way (cf.~\cite[\S{}I.2.2, \S{}I.5.2]{Kn91}).
    We let $O(P,[\beta])$ denote the group of  isometries  of $(P,[\beta])$.
    The category of \emph{unimodular}  quadratic spaces over $(A,\sigma,u,\Lambda)$
    (with isometries as morphisms)
    is denoted by $\Quad[u,\Lambda]{A,\sigma}$.

    \begin{remark}
        When $2\in \units{A}$, we have $\Lambda^{\min}=\Lambda^{\max}$, so there
        is only one form parameter $\Lambda$. Furthermore, in this case,
        $\Quad[u,\Lambda]{A,\sigma}$ is isomorphic to the
        the category of unimodular $u$-hermitian forms over $(A,\sigma,u)$.
        Indeed, $[\beta]$
        can be recovered from $h_\beta$ via $[\beta]=[\frac{1}{2}h_\beta]$.
    \end{remark}

    \begin{remark}
        The ``classical'' notion of  a quadratic space
        over a commutative ring $A$ (see \cite{Roy68}, for instance) occurs in the special case $\sigma=\id_A$, $u=1$ and $\Lambda=0$.
        Then, the quadratic map $\hat{\beta}$ determines $h_\beta$ by \eqref{EQ:beta-h-connection}.
    \end{remark}

    \begin{remark}\label{RM:herm-category}
        The triple $(\rproj{A},*,\{\omega_P\}_{P\in\rproj{A}})$
        is a \emph{hermitian category} and the pair $(1,\underline{\Lambda}:=\{\Lambda_P\}_{P\in\rproj{A}})$
        is a \emph{form parameter} (see \cite{QuSchSch79} or \cite[\S{}II.2]{Kn91} for the definitions).
        There is a one-to-one correspondence between quadratic spaces over $(A,\sigma,u,\Lambda)$
        and quadratic spaces over $(\rproj{A},1,\underline{\Lambda})$
        given by $(P,[\beta])\mapsto (P,[L_\beta])$.
    \end{remark}

\subsection{Conjugation and Transfer}

    We now introduce two well-known procedures that we refer to as \emph{conjugation} and
    \emph{$e$-transfer}. They allow one to alter a unitary ring $(A,\sigma,u,\Lambda)$
    while maintaining data about isometries between quadratic forms (isometry
    groups in particular). We shall use these manipulations several times in the sequel.

    \begin{prp}[``Conjugation'']\label{PR:conjugation-of-unitary-rings}
        Let $v\in\units{A}$. Define
        $(\sigma',u',\Lambda')$ by
        \[
        a^{\sigma'}=va^\sigma v^{-1},\qquad u'=v(v^{\sigma})^{-1}u,\qquad \Lambda'=v\Lambda\ .
        \]
        We call $(\sigma',u',\Lambda')$ the \emph{conjugation} of $(\sigma,u,\Lambda)$ by $v$.
        Then $(A,\sigma',u',\Lambda')$ is a unitary ring and $\Quad[u,\Lambda]{A,\sigma}\cong \Quad[u',\Lambda']{A,\sigma'}$.
    \end{prp}

    \begin{proof}
        This proposition is essentially \cite[Lm.~1.6]{Reiter75}; everything follows by  straightforward computation.
        The categorical equivalence is constructed as follows:
        For any $P\in\rproj{A}$ and $\beta\in\Sesq{P}$, define $v\beta:P\times P\to A$
        by
        \[
        (v\beta)(x,y)=v\cdot \beta(x,y)\ .
        \]
        Then $(P,v\beta)$ is s sesquilinear
        form over $(A,\sigma')$, and the
        assignment $(P,[\beta])\mapsto (P,[v\beta])$ defines  the required categorical isomorphism (isometries are mapped to themselves).
    \end{proof}

    \begin{prp}[``$e$-transfer'']\label{PR:transfer-special-case}
        Let $e\in A$ be an idempotent satisfying $e^\sigma=e$ and $AeA=A$.
        For every sesquilinear form $\beta: P\times P\to A$, denote by $\beta_e$ the restriction
        of $\beta$ to $Pe\times Pe$.
        Then:
        \begin{enumerate}
            \item[(i)] $(B,\tau,v,\Gamma):=(eAe,\sigma|_{eAe},eu,e\Lambda e)$
            is a unitary ring.
            \item[(ii)] $(Pe,[\beta_e])$ is a quadratic space over $(B,\tau,v,\Gamma)$.
            It is unimodular if and only if $(P,[\beta])$ is unimodular.
            \item[(iii)] The map $[\beta]\mapsto [\beta_e]:\Sesq{P}/\Lambda_P\to\Sesq{Pe}/\Lambda_{Pe}$
            is an abelian group isomorphism. In particular, $[\beta]=[0]$ $\iff$ $[\beta_e]=[0]$.
            \item[(iv)] The assignment $(P,[\beta])\mapsto (Pe,[\beta_e])$, called \emph{$e$-transfer},
            gives rise to an equivalence of categories $\Quad[u,\Lambda]{A,\sigma}\sim\Quad[v,\Gamma]{B,\tau}$.
        \end{enumerate}
    \end{prp}

    \begin{proof}[Proof (sketch)]
        Part (i) is straightforward.

        For parts (ii), (iii) and (iv), view $\rproj{A}$ and $\rproj{B}$
        as \emph{hermitian categories with a form parameter} as in Remark~\ref{RM:herm-category}.
        Let $F:\rproj{A}\to \rproj{B}$ be the functor
        given by $FP=P\otimes_A Ae\cong Pe$. By Morita Theory (see \cite[\S18D]{La99}  for instance),
        $F$ is an equivalence of categories.
        In addition,
        there is a natural isomorphism $\phi_P$
        from $F(P^*)=P^*e$
        to $(FP)^* =\Hom_{eAe}(Pe,eAe)$ given by $\phi(f)\mapsto f|_{Pe}$ (check this for $P=A_A$,
        the general case follows by additivity). It is routine to check that
        $(F,\phi)$ is \emph{strictly duality preserving functor} from $\rproj{A}$ to $\rproj{B}$ (see \cite[\S{}2]{QuSchSch79}).
        This means parts (ii) and (iii) hold tautologically, and part (iv)
        follows from \cite[Lm.~2.1]{QuSchSch79} (for instance).
    \end{proof}

\subsection{Simple Unitary Rings}
\label{subsection:simple-art-unit-rings}

    A unitary ring $(A,\sigma,u,\Lambda)$ is called \emph{simple} if the only ideals of $A$ which are invariant
    under $\sigma$ are $0$ and $A$. It is not hard to show that in this case, $A$ is either simple, or
    $A\cong B\times B^\op$, where $B$ is a simple ring, and $\sigma$ is given by $(a,b^\op)^\sigma=(tbt^{-1},a^\op)$ for
    some $t\in\units{B}$.

\medskip

    Assume $(A,\sigma,u,\Lambda)$ is simple and $A$ is artinian. Then
    the Artin-Wedderburn Theorem implies that $A\cong \nMat{D}{n}$ where $D$ is a division ring or a product
    of a division ring and its opposite. Identifying
    $A$ with $\nMat{D}{n}$, we say that such $(A,\sigma,u,\Lambda)$  is \emph{standard}
    or \emph{in standard form}
    if:
    \begin{enumerate}
        \item[(1)] $\sigma$ is of the form $(d_{ij})_{i,j}\mapsto (d_{ji}^\tau)_{i,j}$ for some \emph{involution}
        $\tau:D\to D$. (In particular, $\sigma$ is an involution.)
        \item[(2)] When $D\cong E\times E^\op$ with $E$ a division ring, $\tau$ is the exchange involution
        $(a,b^\op)\mapsto (b,a^\op)$.
        \item[(3)] $u=1$ if  $\tau\neq \id_D$.
    \end{enumerate}
    In this case, we have $u\in\{\pm 1\}$ (because when $\tau=\id_D$, we have $u^2=u^\sigma u=1$).

    It is not true that any simple artinian unitary ring is isomorphic to a unitary ring in standard
    form. However, this is true after applying a suitable conjugation in the sense of Proposition~\ref{PR:conjugation-of-unitary-rings},
    and conjugation does not essentially change the category of quadratic spaces.

    \begin{prp}\label{PR:conjugation-to-standard-form}
        After  a suitable conjugation (cf.\ Proposition~\ref{PR:conjugation-of-unitary-rings}), any simple
        artinian unitary ring $(A,\sigma,u,\Lambda)$ is isomorphic to a  unitary ring in standard form.
    \end{prp}

    \begin{proof}
        Apart from a small difference in condition (3), this proposition is \cite[Pr.~2.1]{Reiter75}.
        We have included here a full proof of the sake of completeness.

        Let $A=\nMat{D}{n}$ be as above.
        By \cite[Th.~7.8]{Fi13A}, $\sigma$ is conjugate
        to some $\sigma'$ of the form $(d_{ij})_{i,j}\mapsto (d_{ji}^\tau)_{i,j}$ where
        $\tau:D\to D$ is an
        \emph{anti-automorphism}, so assume $\sigma$ is in this form.
        This implies that $u$ commutes with the standard matrix units $\{e_{ij}\}$ (because they satisfy
        $e_{ij}^{\sigma\sigma}=e_{ij}$), hence we may view $u$ as an element of $D$ (embedded diagonally
        in $A=\nMat{D}{n}$) which satisfies
        $d^{\tau\tau}=udu^{-1}$ for all $d\in D$. Now, it is enough to show that $(D,\tau)$ can
        be made standard by conjugation.

        Assume that there exists $v\in \units{D}$
        with $u^{-1}v^\tau =v$. Then $u':=v(v^\tau)^{-1}u=1$, so by
        Proposition~\ref{PR:conjugation-of-unitary-rings} (or by
        computation), $\tau':d\mapsto vd^\tau v^{-1}$ is an involution.
        Observe that $d\mapsto u^{-1}d^\tau$ is an involutary additive map,
        hence $v=d+u^{-1}d^\tau$ always satisfies $u^{-1}v^\tau =v$.
        If $D= E\times E^\op$ with $E$ a division ring, take $d=(1_E,0_E^\op)$ to get
        $v\in \units{D}$. Otherwise, any $d$ with $v=d+u^{-1}d^\tau\neq 0$ will do. If
        such $d$ does not exist, then $d^\tau=-ud$ for all $d\in D$. Taking $d=1$ implies
        $u=-1$ and hence $\tau=\id_D$.
        Thus, either $(u,\tau)$ can be conjugated to $(1,\tau')$ with $\tau':D\to D$ an involution,
        or $\tau=\id_D$ and $u=-1$ (in which case $D$ is a field).
        This implies (1) and (3).

        It remains to check (2). Indeed, when $D$ is not a division ring, there is an \emph{isomorphism}
        $D\cong E\times E^\op$
        and \emph{under that isomorphism}
        $\tau$ is given by $(a,b^\op)^\tau=(tbt^{-1},a^\op)$ for some $t\in E$. Since
        $\tau$ is an involution, it must be the exchange involution.
    \end{proof}

    \begin{remark}
        Proposition~\ref{PR:conjugation-to-standard-form} is the reason why many authors
        require $\sigma$ to be an involution in the definition of unitary rings.
        The author does not know if there exists a similar result
        for semilocal rings (i.e.\ a statement
        guaranteeing that $\sigma$ can always be conjugated into an involution).
        See  \cite[Rm.~7.7]{Fi13B} for further discussion.
    \end{remark}

    \begin{prp}\label{PR:factorization-of-unitary-rings}
        Let $(A,\sigma,u,\Lambda)$ be a unitary ring such that $A$ is a semisimple (artinian)
        ring. Then $(A,\sigma,u,\Lambda)$ factors into a product
        \[
        (A,\sigma,u,\Lambda)\cong \prod_{i=1}^t(A_i,\sigma_i,u_i,\Lambda_i):=\Big(\prod_iA_i,\prod_i\sigma,(u_i)_i,\prod_i\Lambda_i\Big)
        \]
        with each $(A_i,\sigma_i,u_i,\Lambda_i)$ simple artinian.
    \end{prp}

    \begin{proof}
        See \cite[p.~486]{Reiter75}, for instance.
    \end{proof}

\subsection{Orthogonality}
\label{subsection:orthogonality}

    We now define a notion of orthogonality for simple artinian unitary rings which will be used
    later in the text (compare with the orthogonality defined in \cite[Ch.~4,~\S2]{Bass73AlgebraicKThyIII} in the commutative case).
    This notion is used implicitly and repeatedly in \cite{Reiter75}.

    \begin{dfn}\label{DF:orthogonal}
        A simple artinian unitary ring $(A,\sigma,u,\Lambda)$ is called
        \emph{orthogonal} if:
        \begin{enumerate}
            \item[(1)] $A$ is simple and of finite dimension over its center, denoted $K$,
            \item[(2)] $\sigma|_K=\id_K$,
            \item[(3)] $\Lambda$ is a $K$-vector space and $\dim_K\Lambda=\frac{1}{2}n(n-1)$
            where $n=\sqrt{\dim_KA}$.
        \end{enumerate}
        If in addition $A\cong \nMat{K}{n}$ (i.e.\ $A$ is split as a central simple $K$-algebra), then we say
        that $(A,\sigma,u,\Lambda)$ is \emph{split-orthogonal}.
    \end{dfn}

    \begin{remark}
        We use the term ``orthogonal'' because
        isometry groups of unimodular quadratic forms over an orthogonal unitary ring $(A,\sigma,u,\Lambda)$
        are \emph{forms} of the the orthogonal group ${\mathbf O}_m(K)$, when viewed as algebraic groups over $K:=\Cent(A)$.
        This follows from the discussion  in~\ref{subsection:dickson} below.
        (A \emph{symplectic unitary ring}
        can likewise be defined by replacing $\frac{1}{2}n(n-1)$ with $\frac{1}{2}n(n+1)$ in condition (3).)
    \end{remark}

    \begin{example}\label{EX:standard-orthogonal}
        If $(A,\sigma,u,\Lambda)$ is simple artinian and \emph{in standard from}
        (see~\ref{subsection:simple-art-unit-rings}), then
        it is split-orthogonal
        if and only if $A\cong\nMat{K}{n}$ for a field $K$, $\sigma$ is the matrix transposition, $u=1$,
        and $\Lambda=\Lambda^{\min}$.
    \end{example}

    Generalizing the example, let $(A,\sigma,u,\Lambda)$ be a unitary ring such that $\sigma$ is an involution.
    If $(A,\sigma,u,\Lambda)$ satisfies conditions (1) and (2), then $A$ is a \emph{central simple algebra} over its center $K$
    (see \cite[\S1]{InvBook}) and $\sigma$ is an involution \emph{of the first kind} (i.e.\ it fixes $\Cent(A)$).
    This easily implies $u\in\{\pm 1\}$.
    By \cite[Pr.~2.6]{InvBook}, when $\Char K\neq 2$, there is $\veps\in\{\pm 1\}$ such that
    \[
    \dim_K\{a\pm a^\sigma\where a\in A\}=\frac{1}{2}n(n\pm \veps) 
    \]
    where $n=\deg A:=\sqrt{\dim_KA}$.
    When $\veps=1$ (resp.\ $\veps=-1$) $\sigma$ is called \emph{orthogonal} (resp.\ \emph{symplectic}).
    Furthermore, when $\Char K=2$, we always have
    \[
    \dim_K\Lambda^{\min}=\dim_K\{a-a^\sigma\where a\in A\}=\frac{1}{2}n(n-1)\ .
    \]
    Thus, when $\sigma$ is an involution, condition (3) is equivalent to having one of the following:
    \begin{enumerate}
        \item[(3a)] $\Char K\neq 2$, $\sigma$ is orthogonal and $u=1$,
        \item[(3b)] $\Char K\neq 2$, $\sigma$ is symplectic and $u=-1$,
        \item[(3c)] $\Char K=2$ and $\Lambda=\Lambda^{\min}$.
    \end{enumerate}
    See \cite[\S2]{InvBook} for further details about orthogonal and symplectic involutions.

    We further recall that the \emph{index} of a central simple $K$-algebra $A$ admitting an
    involution of the first kind is a power of $2$
    (\cite[Cr.~2.8]{InvBook}).
    Thus, if $\deg A$ is odd, then $A$ is split (i.e.\ $\ind A=1$).

    \begin{prp}\label{PR:orth-not-affected-by-conj}
        Orthogonality (resp.\ split-orthogonality) of simple artinian unitary rings is preserved under conjugation
        (see Proposition~\ref{PR:conjugation-of-unitary-rings}).
        Furthermore, if $e\in A$ is an idempotent satisfying $e^\sigma=e$,
        then $(A,\sigma,u,\Lambda)$ is orthogonal (resp.\ split-orthogonal) if and only
        if $(eAe,\sigma|_{eAe},ue,e\Lambda e)$ is orthogonal (resp.\ split-orthogonal).
    \end{prp}

    \begin{proof}
        That orthogonality (resp.\ split-orthogonality) is invariant under conjugation is clear from the definitions, so we turn
        to prove the second statement. Note that since $(A,\sigma,u,\Lambda)$ is simple and $e^\sigma=e$, we
        have
        $AeA=A$ (because $(AeA)^\sigma=AeA$). Morita Theory (see \cite[\S18D]{La99}, for instance)
        now implies that $A$
        is simple if and only if $eAe$ is simple, and $\Cent(eAe)=e\Cent(A)$.
        Writing $K=\Cent(A)$, it follows that
        $A$ is a (split) central simple $K$-algebra if and only if $eAe$ is.
        Furthermore, in this case, it is easy to see that $\sigma$ is of the first kind if and only if $\sigma|_{eAe}$ is.
        Therefore, we may assume $A$ is a central simple $K$-algebra and $\sigma$ is of the first
        kind.

        We claim that $\Lambda$ is a $K$-vector space if and only if $e\Lambda e$ is a $K$-vector space. (In fact, this is
        clear when $\Char K\neq 2$ because $\Lambda=\Lambda^{\min}$ and $\sigma$ is of the first kind.)
        One direction is evident so we turn to show the other. Assume $e\Lambda e$ is a $K$-vector space
        and let $a\in\Lambda$ and $k\in K$. Write $1_A=\sum_ix_iey_i$ for $\{x_i,y_i\}_i\subseteq A$. Then
        $a=(\sum_ix_iey_i)^\sigma a(\sum_j x_jey_j)=\sum_{i,j}a_{ij}$
        where $a_{ij}=y_i^\sigma ex_i^\sigma ax_j ey_j$.
        Observe that since $a^\sigma u=-a$,
        $a_{ji}=-a_{ij}^\sigma u$. Since $e\Lambda e$ is a $K$-vector space,
        $k\cdot ex_i^\sigma a x_ie\in e\Lambda e\subseteq\Lambda$ for all $i$,
        hence $k\cdot a_{ii}=y_i^\sigma (kex_i^\sigma a x_ie)y\in\Lambda$. Now,
        $ka=\sum_{i,j}a_{ij}=k\sum_{i<j}(a_{ij}-a_{ij}^\sigma u)+k\sum_i a_{ii}
        =\sum_{i<j}(ka_{ij}-(ka_{ij})^\sigma u)+\sum_i (ka_{ii})
        \in \Lambda$,
        as required.

        Assume $\Lambda$ is a $K$-vector space.
        It is left to show that $\dim_K\Lambda=\frac{1}{2}n(n- 1)$ if and only if
        $\dim_Ke\Lambda e=\frac{1}{2}m(m-1)$, where $n=\deg A$ and $m=\deg eAe$.
        Observe that by the above discussion, when $\sigma$ is an involution, we always have
        $\dim_K\Lambda=\frac{1}{2}n(n\pm 1)$ and $\dim_Ke\Lambda e=\frac{1}{2}m(m\pm 1)$.
        Thus, by Proposition~\ref{PR:conjugation-to-standard-form}, the same holds for arbitrary $\sigma$.
        Let $f=1-e$. There is nothing to prove if $f=0$.
        Otherwise, $\deg fAf=n-m$, hence $\dim_K f\Lambda f=\frac{1}{2}(n-m)(n-m\pm 1)$.
        It is easy to check that $\dim \Lambda =\dim e\Lambda e+\dim f\Lambda f+\dim eAf=
        \dim e\Lambda e+\dim f\Lambda f+m(n-m)$, and this implies
        $\dim \Lambda=\frac{1}{2}n(n-1)$ if and only
        if $\dim e\Lambda e=\frac{1}{2}m(m-1)$.
    \end{proof}

\subsection{Semiperfect Rings}
\label{subsection:semiperfect}

    We finish this section with recalling several facts about semiperfect rings.
    Proofs  and additional details can be found in \cite[\S2.7--\S2.9]{Ro88}.

\medskip

    A ring $A$ is called \emph{semiperfect} if it satisfies the following equivalent
    conditions:
    \begin{enumerate}
        \item[(a)] $A$ is is semilocal and $\Jac(A)$ is idempotent lifting.
        \item[(b)] All finitely generated right (or left) $A$-modules have  a \emph{projective cover}
        (see \cite[Df.~2.8.31]{Ro88}).
        \item[(c)]
        There exists orthogonal idempotents $e_1,\dots,e_n\in A$ with $\sum_ie_i=1$
        and such that $e_iAe_i$ is local for all $i$.
    \end{enumerate}
    In this case, any system of orthogonal idempotents in $A/\Jac(A)$ can
    be lifted to a system of orthogonal idempotents in $A$. Furthermore,
    $eAe$ is semiperfect for any idempotent $e\in A$.

\medskip

    Examples of semiperfect rings include all
    one-sided artinian rings, and more gerenally, all semilocal rings $A$ with $A=\invlim \{A/\Jac(A)^n\}_{n\in\N}$.
    Further examples that will be used later can be obtained from the following proposition.

    \begin{prp}\label{PR:semiperfect-examples}
        Let $R$ be a henselian local (commutative) ring, and let $A$ be an $R$-algebra. Then
        $A$ is semiperfect if one of the following holds:
        \begin{enumerate}
            \item[(1)] $R$ is noetherian and $A$ is finitely generated as an $R$-module.
            \item[(2)] $R$ is a valuation ring and $A$ is $R$-torsion-free and of finite rank over $R$.\footnote{
                For an integral domain $R$, the rank an $R$-module $M$ is $\dim_KM\otimes_RK$,
                where $K$ is the fraction field of $R$.
            }
        \end{enumerate}
    \end{prp}

    \begin{proof}
        When (1) holds, this follows from \cite[Th.~22]{Azu51} or \cite[Lm.~12]{Vamos90}.
        When (2) holds, $A$ is semilocal by
        \cite[Th.~5.4]{Warfield80}. Let $e\in A$ be an idempotent such that $eAe$ has no idempotents
        other than $e$ and $0$. The proof of \cite[Lm.~14]{Vamos90} then implies that $eAe$ is local.
        Replacing $A$ with $(1-e)A(1-e)$ and repeating this procedure yields a (finite) system of orthogonal
        idempotents $e=e_1,e_2\dots,e_t\in A$ with $\sum_ie_i=1$ and such that $e_iAe_i$ local for all $i$,
        so $A$ is semiperfect.
    \end{proof}

    Let $A$ be a semiperfect ring.
    Then, up to isomorphism, there exist finitely many indecomposable projective
    $A$-modules, $\{P_i\}_{i=1}^t$, and every $P\in\rproj{A}$ can be written as $P\cong \bigoplus_{i=1}^t P_i^{n_i}$ with $(n_i)_{i=1}^t$
    uniquely determined. If $\quo{A}=A/\Jac(A)$,
    then $\{\quo{P}_i:=P_i/P_i\Jac(A)\}_{i=1}^t$ are the simple $\quo{A}$-modules,
    up to isomorphism.

    The modules $P_1,\dots,P_t$ can be constructed as follows:
    Write $\quo{A}$ as a product of simple artinian rings $\prod_{i=1}^t A_i$, let $\veps_i$
    be a primitive idempotent in $A_i$, and let $e_i$ be a lifting of $\veps_i$ to $A$.
    Then $e_1A,\dots,e_tA$ are the indecomposable projective right $A$-modules, up to isomorphism.

\section{Reflections and Quasi-Reflections}
\label{section:reflection}

    In this section we introduce and study \emph{quasi-reflections}, which slightly extend a notion of  reflections
    used by Reiter \cite{Reiter75}.
    Throughout,  $(A,\sigma,u,\Lambda)$ is a unitary ring, and
    $(P,[\beta])$ be a quadratic space over $(A,\sigma,u,\Lambda)$.

\medskip

    Let $e,f\in A$ be idempotents. An element $a\in eAf$ is called
    $(e,f)$-invertible if there exists $a'\in fAe$ such that $aa'=e$ and $a'a=f$.
    It is easy to see that $a'$ is unique and has $a$ as its $(f,e)$-inverse. We hence
    write $a'=\efinv{e}{f}{a}$, or just $a'=a^\circ$ when $e,f$ are understood from the context.
    Notice that there exists an $(e,f)$-invertible element if and only $eA\cong fA$,
    in which
    case we write $e\sim f$.
    Indeed, left multiplication by an $(e,f)$-invertible element gives an isomorphism from $fA$ to $eA$,
    and any isomorphism $fA\to eA$ is easily seen to be of this form.
    
    \begin{lem}\label{LM:ef-inv-mod-Jac}
        Let $e,f\in A$ be idempotents, and set $\quo{x}:=x+\Jac(A)\in A/\Jac(A)$ for all $x\in A$.
        Then $a\in eAf$ is $(e,f)$-invertible
        if and only if $\quo{a}$ is $(\quo{e},\quo{f})$-invertible.
    \end{lem}

    \begin{proof}
        We only show the non-trivial direction. Assume $\quo{a}$ has an $(\quo{e},\quo{f})$-inverse $\quo{b}$
        with $b\in fAe$. Then $\quo{ab}\in\units{(\quo{eAe})}$, hence $ab\in \units{eAe}$.
        (This follows from the easy fact  that $\Jac(A)\cap eAe=e\Jac(A)e=\Jac(eAe)$.)
        Likewise,
        $ba\in \units{fAf}$. Let $c$ be the inverse of $ab$ in $eAe$, and let $d$ be the inverse
        of $ba$ in $fAf$. Then $a(bc)=e$, $(db)a=f$, and $db=(db)e=(db)a(bc)=f(bc)=bc$. Thus,
        $a'=bc=db$ is an $(e,f)$-inverse of $a$.
    \end{proof}

    Let $e\in A$ be an idempotent, let $y\in Pe$, and let $c\in e^{\sigma}\hat{\beta}(y)e={\beta}(y,y)+e^\sigma \Lambda e$ be $(e^\sigma,e)$-invertible.
    We define $s_{y,e,c}:P\to P$ by
    \[
    s_{y,e,c}(x)=x-y\cdot c^\circ \cdot {h}_\beta(y,x)\ .
    \]
    Observe that $s_{y,e,c}$ is completely determined by the class $[\beta]$.
    Following \cite{Reiter75}, we call $s_{y,e,c}$ an \emph{$e$-reflection}.
    We will also use the name \emph{quasi-reflection}, which
    does not  restrict us to a particular idempotent $e$. A \emph{reflection} of $(P,[\beta])$
    is a $1$-reflection. 
    When we want to stress
    the quadratic form $[\beta]$, we shall write $s_{y,e,c}^{[\beta]}$ instead of $s_{y,e,c}$.

    \begin{remark}
        Reiter's definition of $e$-reflections  (\cite[Df.~1.2]{Reiter75}) is essentially the same, except
        that he assumes $e=e^\sigma$ (in which case $c^\circ$ is just the inverse of $c$
        in $eAe$). The generalization defined here will play a crucial role later in the text.
    \end{remark}

    \begin{prp}
        In the previous setting, $s_{y,e,c}$ is an isometry of $(P,[\beta])$. Its inverse
        is $s_{y,e,c^\sigma u}$.
    \end{prp}

    \begin{proof}
        This is similar to
        the proof of  \cite[Pr.~1.3]{Reiter75}; replace the usual inverses with $(e^\sigma,e)$-inverses.
    \end{proof}

    \begin{lem}\label{LM:reflections-manipulations}
        Let $e,f\in A$ be  idempotents.
        \begin{enumerate}
            \item[(i)] If $e\sim f$, then $e$-reflections and $f$-reflections coincide.
            \item[(ii)] If $ef=fe=0$, then the composition of an $e$-reflection and an $f$-reflection is an $(e+f)$-reflection.
            Specifically, $s_{y,e,c}s_{z,f,d}=s_{y+z,e+f,c+d+{h}(y,z)}$.
        \end{enumerate}
    \end{lem}

    \begin{proof}
        (i) Let $a\in eAf$ be an $(e,f)$-invertible element. It is a straightforward computation
        to verify that $s_{y,e,c}=s_{ya,f,a^\sigma ca}$, which proves the claim.

        (ii) Throughout the proof, we
        shall make repeated implicit usage of the fact that $yf=ze=0$ and
        $fc^\circ=c^\circ f^\sigma=ed^\circ=d^\circ e^\sigma=0$,
        which easily follows from $ef=fe=0$.

        Observe first that
        \begin{eqnarray*}
        c+d+{h}(y,z) &\in & e^\sigma\hat{\beta}(y)e+f^\sigma\hat{\beta}(z)f+e^\sigma{h}(y,z)f\\
        &\subseteq&
        (e+f)^\sigma(\hat{\beta}(z)+\hat{\beta}(y)+{h}(y,z))(e+f)\\
        &=&(e+f)^\sigma\hat{\beta}(z+y)(e+f)
        \end{eqnarray*}
        and that $c+d+{h}(y,z)$ is $((e+f)^\sigma,e+f)$-invertible with inverse
        $c^\circ+d^\circ -c^\circ {h}(y,z)d^\circ$. Thus, $r:=s_{y+z,e+f,c+d+{h}(y,z)}$
        is an $(e+f)$-reflection. Now,
        for all $x\in P$, we have
        \begin{eqnarray*}
            rx &=& x- (y+z)(c^\circ+d^\circ-c^\circ{h}(y,z)d^\circ){h}(y+z,x)\\
            &=& x-(yc^\circ +zd^\circ-yc^\circ{h}(y,z)d^\circ)({h}(y,x)+{h}(z,x))\\
            &=& x-zd^\circ{h}(z,x)-yc^\circ{h}(y,x)+y{c^\circ}{h}(y,z)d^\circ{h}(z,x)\\
            &=& (x-zd^\circ{h}(z,x))-yc^\circ{h}(y,x-zd^\circ{h}(z,x))\\
            &=& s_{y,e,c}(s_{z,f,d}x)\ ,
        \end{eqnarray*}
        as required.
    \end{proof}

    \begin{lem}\label{LM:action-of-reflections}
        Let $e\in A$ be an idempotent and
        let $x,y\in Pe$.
        \begin{enumerate}
            \item[(i)] If $c:=h_\beta(x-y,x)$ is $(e^\sigma,e)$-invertible,
            then $s_{x-y,e,c}(x)=y$.
            \item[(ii)] If there exist $z\in Pe$ and $(e^\sigma,e)$-invertible
            $c\in\hat{\beta}(z)$ such that $d:=h_\beta(y,w)$ is $(e^\sigma,e)$-invertible
            for $w=y-s_{z,e,c}(x)=y-x+zc^\circ {h}_\beta(z,x)$, then
            $s_{w,e,d}s_{z,e,c}(x)=y$.
        \end{enumerate}
    \end{lem}

    \begin{proof}
        This is essentially the same as the proofs of Lemma~1.4 and Lemma~1.5
        in \cite{Reiter75}; replace the usual inverses with $(e^\sigma,e)$-inverses.
    \end{proof}

    \begin{remark}\label{RM:conjugation-does-not-affect-reflections}
        (i) When applying conjugation (see Proposition~\ref{PR:conjugation-of-unitary-rings})
        with respect  to $v\in \units{A}$,
        $e$-reflections remain $e$-reflections. Indeed, it is straightforward to check
        that $s^{[\beta]}_{y,e,c}=s^{[v\beta]}_{y,e,vc}$ (note that if
        $a^{\sigma'}:=va^\sigma v^{-1}$, then $\efinv{e^{\sigma'}}{e}{vc}=c^\circ v^{-1}$).
        (In Reiter's setting, which assumes $e=e^\sigma$, $e$-reflections
        are preserved only when $e$ commutes with $v$, for otherwise, $e$ is not invariant
        under the conjugation of $\sigma$ by $v$.)

        (ii) Let $e,f\in A$ be  idempotents with $e=e^\sigma$ and  $f\in eAe$. Then $e$-transfer
        (see Proposition~\ref{PR:transfer-special-case})
        sends $f$-reflections of $(P,[\beta])$ to $f$-reflections of $(Pe,[\beta_e])$.
        Indeed, $s_{y,e,c}^{[\beta]}|_{Pe}=s_{y,e,c}^{[\beta_e]}$.
    \end{remark}

\section{Witt's Extension Theorem}
\label{section:Witt-thm}

    Using  methods of Reiter \cite{Reiter75} and the notion of quasi-reflections above, we
    now show that every isometry between subspaces of a unimodular quadratic space over a semiperfect
    ring can be extended to an isometry of the whole quadratic space.
    Furthermore, with small exception, the resulting isometry is a product of quasi-reflections.
    We compare our results with those of Reiter in Remark~\ref{RM:comparison}.

\subsection{General Setting}
\label{subsection:general-setting}

    We set some general notation that will be used throughout.
    Let $(A,\sigma,u,\Lambda)$ be a semiperfect unitary ring.
    For $P\in\rproj{A}$, set $\quo{P}:=P/P\Jac(A)$. In particular, $\quo{A}=A/\Jac(A)$.
    We shall occasionally view $\quo{P}$ as a right $\quo{A}$-module.
    The image of $x\in P$ in $\quo{P}$ will be denoted by $\quo{x}$.
    Note that $P,Q\in\rproj{A}$
    are isomorphic if and only if $\quo{P}\cong\quo{Q}$, because $P$ is a \emph{projective
    cover} of $\quo{P}$ and projective covers are unique up to isomorphism.

\medskip

    Let $\quo{\Lambda}=\{\quo{\lambda}\where \lambda\in\Lambda\}$ and let $\quo{\sigma}$
    be the map induced by $\sigma$ on $\quo{A}$. Then $(\quo{A},\quo{\sigma},\quo{u},\quo{\Lambda})$
    is a semisimple unitary ring, hence, by Proposition~\ref{PR:factorization-of-unitary-rings}, it factors into a product
    \[
    (\quo{A},\quo{\sigma},\quo{u},\quo{\Lambda})\cong\prod_{i=1}^\ell(A_i,\sigma_i,u_i,\Lambda_i)
    \]
    with each $(A_i,\sigma_i,u_i,\Lambda_i)$ a simple artinian unitary ring (see~\ref{subsection:simple-art-unit-rings}).
    We write $A_i=\nMat{D_i}{n_i}$ with $D_i$ a division ring, or a product of a division
    ring and its opposite.

\medskip

    By Proposition~\ref{PR:conjugation-to-standard-form}, for every $1\leq i\leq \ell$, there exists
    $v_i\in\units{A_i}$ such that the conjugation (see Proposition~\ref{PR:conjugation-of-unitary-rings})
    of
    $(\sigma_i,u_i,\Lambda_i)$ by  $v_i$ is \emph{standard} (see~\ref{subsection:simple-art-unit-rings}).
    Choose $v\in \units{A}$ whose image in $A_i$ is $v_i$.
    Then, by conjugating   $(\sigma,u,\Lambda)$ with $v$,
    we may assume that $(A_i,\sigma_i,u_i,\Lambda_i)$ is in standard form for all $i$.
    Note that conjugation preserves quasi-reflections by
    Remark~\ref{RM:conjugation-does-not-affect-reflections}(i), so this
    is allowed if our goal is to prove that certain isometries extend to  a product of quasi-reflections.

    Now, let $\veps_i$ denote the standard matrix unit $e_{11}$ in $A_i$. Then $\veps_1,\dots,\veps_\ell$
    are orthogonal $\quo{\sigma}$-invariant idempotents. Since $\Jac(A)$ is idempotent lifting,
    we can lift $\veps_1,\dots,\veps_\ell$ to orthogonal idempotents $e_1,\dots,e_\ell\in A$.
    The idempotents $e_i$ may not be invariant under $\sigma$, but we have $e_iA\cong e_i^\sigma A$
    as right $A$-modules (because $\quo{e_iA}=\quo{e_i^\sigma A}$), and hence $e_i\sim e_i^\sigma$.

    Next, we set
    \[
    (A_{(i)},\sigma_{(i)},u_{(i)},\Lambda_{(i)})=(\veps_iA_i\veps_i,\sigma_i|_{\veps_iA_i\veps_i},\veps_iu_i,\veps_i\Lambda_i\veps_i)\ .
    \]
    Note that by Proposition~\ref{PR:orth-not-affected-by-conj}, $(A_i,\sigma_i,u_i,\Lambda_i)$
    is split-orthogonal if and only if $(A_{(i)},\sigma_{(i)},u_{(i)},\Lambda_{(i)})$ is split-orthogonal,
    and in this case, $A_{(i)}$ is a field, $\sigma_{(i)}=\id_{A_{(i)}}$, $u_{(i)}=1$ and $\Lambda_{(i)}=0$.
    Also note that $A_{(i)}\cong D_i$.

\medskip

    Let $(P,[\beta])$ be a quadratic space over $(A,\sigma,u,\Lambda)$.
    Then $(P,[\beta])$ gives rise to a quadratic space $(\quo{P},[\quo{\beta}])$
    over $(\quo{A},\quo{\sigma},\quo{u},\quo{\Lambda})$; the map $\quo{\beta}$ is defined
    by
    \[
    \quo{\beta}(\quo{x},\quo{y})=\quo{\beta(x,y)}\qquad\forall\, x,y\in P\ .
    \]
    Since $(\quo{A},\quo{\sigma},\quo{u},\quo{\Lambda})$ factors into  a product
    of unitary rings, the datum of $(\quo{P},[\quo{\beta}])$ is equivalent
    to the datum of  quadratic spaces $(P_i,[\beta_i])_{i=1}^\ell$
    over $(A_i,\sigma_i,u_i,\Lambda_i)_{i=1}^\ell$. Specifically, if we write $\quo{P}=\prod_{i=1}^\ell P_i$
    with $P_i$ a right $A_i$-module,
    then $\beta_i$ is just the restriction of $\beta$ to the copy of $P_i$ in $\quo{P}$.
    We further set $P_{(i)}=P_i\veps_i\in\rproj{A_{(i)}}$ and let $\beta_{(i)}=(\beta_i)_{\veps_i}:=\beta_i|_{P_{(i)}\times P_{(i)}}$.
    Recall from Proposition~\ref{PR:transfer-special-case} that $(P_{(i)},[\beta_{(i)}])$ is a quadratic space over
    $(A_{(i)},\sigma_{(i)},u_{(i)},\Lambda_{(i)})$, which is the $\veps_i$-transfer
    of $(P_i,\beta_i)$. As a result, we have
    \[
    [\beta_i]=[0]\quad\iff\quad [\beta_{(i)}]=[0]\ .
    \]
    Next, let
    \[
    \pi_i:P\to P_i
    \]
    denote the map sending $x\in P$ to its image in $P_i$. We clearly have $\pi_i(Pe_i)=P_{(i)}$
    and $\pi_i\beta(x,y)=\beta_i(\pi_ix,\pi_iy)$.

    Keeping the previous setting,
    let $Q\subseteq P$ be a submodule. Then ${h}_\beta$ induces a map
    \[\lAd{Q}=\lAd{[\beta],Q}:P\to Q^*:=\Hom_A(Q,A)\]
    sending $x\in P$ to $[y\mapsto h_\beta(x,y)]\in Q^*$.
    This map is evidently onto when $(P,[\beta])$ is unimodular and $Q$ is a summand of $P$.
    Lastly,  set
    \[
    Q^\perp=Q^{\perp[\beta]}=\{x\in P\suchthat h_\beta(x,P)=0\}\ ,
    \]
    and let $\beta|_Q$ denote
    the restriction of $\beta$ to $Q\times Q$.

\subsection{Lemmas}

    Before giving the main result, we collect several lemmas.

    \begin{lem}\label{LM:symmetric-modules}
        Let $P\in\rproj{A}$. Then $P\cong P^*$ $\iff$ $P$ is isomorphic to a direct sum of
        copies of the modules $e_1A,\dots,e_\ell A$. (See \ref{subsection:unitary-rings} for the definition of $P^*$.)
    \end{lem}

    \begin{proof}
        It is straightforward to check that $(e_iA)^*\cong e_i^\sigma A$, and
        since $e_i\sim e_i^\sigma$, we have $e_i^\sigma A\cong e_i A$. This settles the direction ``$\Longleftarrow$'',
        so we turn to show the converse.

        Without loss of generality, $D_i$ is a division ring for $i\leq t$,
        and a product of a division ring and its opposite otherwise. In the latter case, $\veps_i$
        can be written as a sum of two orthogonal idempotents $\veps_i=\delta_i+\delta_i^\sigma$.
        Now, $\veps_1,\dots,\veps_t,\delta_{t+1},\dots,\delta_{\ell},\delta_{t+1}^\sigma,\dots,\delta_{\ell}^\sigma$
        is a system of orthogonal idempotents in $\quo{A}$, hence it can be lifted to a system of orthogonal idempotents
        $f_1,\dots,f_{2\ell-t}$ in $A$. Moreover, $f_1A_1,\dots,f_{2\ell-t}A$ are the indecomposable
        projective right $A$-modules, up to isomorphism (see~\ref{subsection:semiperfect}).

        Assume that $P\cong P^*$. Then $P\cong\bigoplus f_iA_i^{n_i}$ with $(n_i)_{i=1}^{2\ell-t}$ uniquely
        determined. It is straightforward to check that $(f_iA_i)^*\cong f_i^\sigma A_i$.
        By the way we have constructed $f_1,\dots,f_{2\ell-t}$, we have $f_i^\sigma\sim f_i$ if $i\leq t$,
        and $f_i^\sigma\sim f_{i+(\ell-t)}$ if $t<i\leq\ell$. Thus, $n_{i}=n_{i+(\ell-t)}$ for all $t<i\leq \ell$.
        Since $e_iA\cong f_i A\oplus f_{i+(\ell-t)}A$ for $t<i\leq \ell$, it follows
        that $P\cong \bigoplus_{i=1}^\ell (e_iA)^{n_i}$, as required.
    \end{proof}

    \begin{lem}\label{LM:onto-is-preserved}
        Let $1\leq i\leq\ell$, let $(P,[\beta])$ be a quadratic space, and let $Q,V\subseteq P$ summands of
        $P$. Then:
        \begin{enumerate}
        \item[(i)] $\lAd{[\beta],Q}(V)=Q^*$ $\derives$ $\lAd{[\beta_{(i)}],Q_{(i)}}(V_{(i)})=Q_{(i)}^*$.
        \item[(ii)] $[\beta]$ is unimodular $\derives$ $[\beta_i]$ and $[\beta_{(i)}]$ are unimodular.
        \item[(iii)]
        $V_i^{\perp[\beta_i]}=V_{(i)}^{\perp[\beta_i]}$ and $V_i^{\perp[\beta_i]}\veps_i=V_{(i)}^{\perp[\beta_{(i)}]}$.
        \end{enumerate}
        (Here, $Q_{(i)}=\pi_i(Q)\veps_i$ and $V_{(i)}=\pi_i(V)\veps_i$.)
    \end{lem}

    \begin{proof}
        (i) Let $\psi\in\Hom_{A_{(i)}}(Q_{(i)},A_{(i)})=\Hom_{\veps_iA_i\veps_i}(Q_i\veps_i,\veps_iA\veps_i)$.
        As explained in the proof of Proposition~\ref{PR:transfer-special-case},
        the functor $U\mapsto U\veps_i:\rproj{A_i}\to \rproj{\veps_iA_i\veps_i}$ is an equivalence,
        hence there exists unique $\psi'\in \Hom_{A_i}(Q_i,\veps_iA_i)$
        with $\psi'|_{Q_{(i)}}=\psi$.
        Since $Q$ is projective, there is $\vphi\in\Hom_A(Q,e_i^\sigma A)$ with $\pi_i\circ \vphi=\psi'\circ\pi_i$.
        By assumption, there is  $x\in V$ such that $\lAd{[\beta],Q}(x)=\vphi$, and
        by replacing $x$ with $xe_i$ we may assume $x=xe_i\in V_ie_i$.
        Now, for all $y\in Qe_i$, we have
        $\lAd{[\beta_{(i)}],Q_{(i)}}(\pi_i x)(\pi_i y)=\beta_{(i)}(\pi_i x,\pi_i y)=
        \beta_i(\pi_ix,\pi_i y)
        =\pi_i(\beta(x,y))=
        \pi_i(\lAd{[\beta],Q}(x)(y))=\pi_i(\vphi y)=\psi (\pi_i y)$,
        so $\lAd{[\beta_{(i)}],Q_{(i)}}(\pi_i x)=\psi$.

        (ii) Putting $Q=V=P$ in (i) implies that $\lAd{[\beta_{(i)}]}$ is onto.
        Since $P_{(i)}^*$ and $P_{(i)}$ have the same length, $\lAd{[\beta_{(i)}]}$ is also injective,
        so $[\beta_{(i)}]$ is unimodular.
        By Proposition~\ref{PR:transfer-special-case}, $[\beta]$ is also unimodular.

        (iii)
        This is routine. Use the fact that $V_{(i)}A_i=V_iA_i\veps_iA_i=V_iA_i=V_i$.
    \end{proof}

    \begin{lem}\label{LM:dual-decomp}
        Let $P\in\rproj{A}$, and assume $P^*=U\oplus U'$.
        Then $P$ factors as a direct sum $Q\oplus Q'$ such
        that $U=\{f\in P^*\suchthat f(Q')=0\}$ and $U'=\{f\in P^*\suchthat f(Q)=0\}$.
    \end{lem}

    \begin{proof}
        Let $\omega=\omega_P:P\to P^{**}$ (see~\ref{subsection:unitary-rings} for the definition).
        We identify $P^{**}$ with $U^*\oplus U'^*$ via $g\mapsto (g|_U,g|_{U'})$.
        Let $Q=\omega^{-1}(U^*)$ and $Q'=\omega^{-1}(U'^*)$. We clearly have $P=Q\oplus Q'$.
        Furthermore, for
        $f\in P^*$, we have $f\in U$ $\iff$ $\psi f=0$ for all $\psi\in U'^*$
        $\iff$ $(\omega x)f=0$ for all $x\in Q'$ $\iff$ $(fx)^\sigma u=0$
        for all $x\in Q'$ $\iff $ $fx=0$ for all $x\in Q'$.
        Thus,
        $U=\{f\in P^*\suchthat f(Q')=0\}$, and likewise, $U'=\{f\in P^*\suchthat f(Q)=0\}$.
    \end{proof}

    \begin{lem}\label{LM:F-two-squared}
        Assume $A\cong \F_2\times \F_2$ and $\sigma$ is the exchange involution
        (this forces $u=1$ and $\Lambda=\{0,1\}$).
        Let $(P,[\beta])$ be a quadratic space over $A$, let $V\subseteq P$ be a submodule,
        and let $x\in P$, $z\in V$ be such that $h_\beta(x,z)=1$. Then the following conditions
        are equivalent:
        \begin{enumerate}
            \item[(a)] There is no $z'\in V$  with $h_\beta(x,z')=1$ and $\hat{\beta}(z')=\Lambda=\{0,1\}$.
            \item[(b)] $h_\beta(z,z)=1$ and $[\beta|_{z^\perp\cap V}]=[0]$.
        \end{enumerate}
        In this case, $\lAd{V}(V)\cong A$.
    \end{lem}

    \begin{proof}
        For the equivalence (a)$\iff$(b), see \cite[Lm.~3.8c]{Reiter75}. (In Reiter's notation, $H_f=H\cap f^\perp$ and $N$ stands for
        the zero form $[0]$.)
        When (b) holds, we have $V=zA\oplus (z^\perp \cap V)$
        since $x= (zh_\beta(z,x))+(x-zh_\beta(z,x))$ for all $x\in V$.
        This means that $\lAd{V}(V)=\lAd{V}(zA)$, which is easily seen to imply $\lAd{V}(V)\cong (zA)^*\cong A^*\cong A$.
    \end{proof}

\subsection{Witt's Extension Theorem}

    We are now in position to phrase and prove an analogue of Witt's Extension Theorem
    (Theorems~\ref{TH:Witt-I} and~\ref{TH:Witt-II}). Following are several immediate (and less technical) corollaries.
    We compare our results with those of Reiter \cite{Reiter75} in Remark~\ref{RM:comparison} below.

    \begin{thm}\label{TH:Witt-I}
        Let $(P,[\beta])$ be a quadratic space, let $Q,S,V$ be summands of
        $P$, and let $\psi:(Q,[\beta|_Q])\to (S,[\beta|_S])$ be an isometry.
        Assume that the following holds:
        \begin{enumerate}
            \item[(1a)] $\lAd{Q}(V)=Q^*$ and $\lAd{S}(V)=S^*$.
            \item[(1b)] $\psi x-x\in V$ for all $x\in Q$.
            \item[(1c)] $Q\cong Q^*$ and $S\cong S^*$.
            \item[(2a)] If
            $(A_i,\sigma_i,u_i,\Lambda_i)$ is split-orthogonal, then  $Q_i=0$ or $[\beta_i|_{V_i}]\neq 0$.
            \item[(2b)] If $D_i\cong \F_2$ and $(A_i,\sigma_i,u_i,\Lambda_i)$ is split-orthogonal,
            then $[\beta_{i}|_{V_{i}\cap V_{i}^\perp}]\neq [0]$.
            \item[(2c)] If $D_i\cong \F_2\times\F_2$, then there is no $z\in V_{(i)}=V_ie_i$ satisfying
            $h_{\beta_{i}}(z,z)=\veps_i$ and $[\beta_{i}|_{z^\perp\cap V_{i}}]=[0]$.
        \end{enumerate}
        Then there is
        $\vphi\in O(P,[\beta])$ such that
        $\vphi$ is a product
        of quasi-reflections taken with respect to elements of $V$,
        and $\vphi|_Q=\psi$. (The former implies $\vphi x-x\in V$ for all $x\in P$.)
    \end{thm}

    We  first prove the following special case.

    \begin{lem}\label{LM:main-lemma}
        Theorem~\ref{TH:Witt-I} holds when $Q\cong e_iA$ for some $1\leq i\leq \ell$.
        In fact, in this case, $\vphi$ can be taken to be a product of $e_i$-reflections.
    \end{lem}

    \begin{proof}
        The proof is based on reduction to the proof of \cite[Th.~4.1]{Reiter75},
        with certain modifications (particularly in the case $D_i\cong \F_2\times\F_2$).
        Throughout, $h=h_\beta$.

        Write $e=e_i$. Since $Q\cong eA$, there is $x\in Q$ with $xe=x$ and $Q=xA$.
        Let $y=\psi x$. It is enough to find a product of $e$-reflections of $(P,[\beta])$, taken
        with respect to elements of $V$, that sends $x$ to $y$.
        We shall prove this by applying Lemma~\ref{LM:action-of-reflections},
        except maybe in the case $D_i\cong \F_2\times \F_2$.

        If $h(x-y,x)$ is $(e^\sigma,e)$-invertible, then $s_{x-y,e,h(x-y,x)}(x)=y$ (and $x-y=x-\psi x\in V$
        by (1b)). If not, then it is enough to find $z\in Ve_i$ and $(e^\sigma, e)$-invertible $c\in \hat{\beta}(z)$ such
        that
        \[\Phi(z,c):=h(y,y-x+zc^\circ h(z,x))=h(y,y-x)+h(y,z)c^\circ h(z,x)\]
        is  $(e^\sigma,e)$-invertible, in which case we shall have $s_{w,e,\Phi(z,c)}s_{z,e,c}(x)=y$
        for $w=(x-y)+zc^\circ h(z,x)\in Ve_i$. Note that
        \[
        h(x-y,x)=h(x,x)-h(y,x)=h(y,y)-h(y,x)=h(y,y-x)\ ,
        \]
        so we may henceforth assume that $h(y,y-x)$ is not $(e^\sigma,e)$-invertible.

        Reducing everything modulo $\Jac(A)$ (using Lemma~\ref{LM:ef-inv-mod-Jac}), we see that it
        is enough to find $\quo{z}\in V_{(i)}=\quo{Ve_i}$ and $\quo{c}\in \hat{\beta}_{(i)}(\quo{z})\cap \units{A_{(i)}}$
        such that $\quo{\Phi}(\quo{z},\quo{c})$ is invertible in $A_{(i)}$, where
        \[
        \quo{\Phi}(\quo{z},\quo{c})=h_{\beta_{(i)}}(\quo{y},\quo{y}-\quo{x})+h_{\beta_{(i)}}(\quo{y},\quo{z})\quo{c}^\circ h_{\beta_{(i)}}(\quo{z},\quo{x})\ .
        \]
        We may assume that $h_{\beta_{(i)}}(\quo{y},\quo{y}-\quo{x})$ is not invertible in $A_{(i)}$.

        Now, conditions (1a) and (2a)--(2c) imply (respectively) that:
        \begin{enumerate}
            \item[{(1a-$i$)}] $L_{Q_{(i)}}(V_{(i)})=Q_{(i)}^*$ and $L_{S_{(i)}}(V_{(i)})=S_{(i)}^*$.
            \item[{(2a-$i$)}] If $(A_{(i)},\sigma_{(i)},u_{(i)},\Lambda_{(i)})$ is split-orthogonal, then $[\beta_{(i)}|_V]\neq 0$.
            \item[{(2b-$i$)}] If $A_{(i)}\cong \F_2$ and $(A_{(i)},\sigma_{(i)},u_{(i)},\Lambda_{(i)})$ is split-orthogonal,
            then $[\beta_{(i)}|_{V_{(i)}\cap V_{(i)}^\perp}]\neq [0]$.
            \item[{(2c-$i$)}] If $A_{(i)}\cong \F_2\times\F_2$, then there is no $\quo{z}\in V_{(i)}$ with
            $h_{\beta_{(i)}}(\quo{z},\quo{z})=1$ and $[\beta_{(i)}|_{\quo{z}^\perp\cap V_{(i)}}]=[0]$ .
        \end{enumerate}
        (Use Lemma~\ref{LM:onto-is-preserved} and Proposition~\ref{PR:orth-not-affected-by-conj}.)
        We are now reduced to the setting
        of steps 2--5 in the proof of \cite[Th.~4.1]{Reiter75} (applied to  $(P_{(i)},[\beta_{(i)}])$
        and $(A_{(i)},\sigma_{(i)},u_{(i)},\Lambda_{(i)})$), in which the existence of
        $\quo{z}$ and $\quo{c}$ as above is shown, except maybe in the case $A_{(i)}\cong\F_2\times\F_2$.

        Assume henceforth that $A_{(i)}\cong\F_2\times\F_2$ (which means $\sigma_{(i)}$ is the exchange involution).
        Let $\veps:=\veps_i=1_{A_{(i)}}$.
        It is shown in cases~2 and~4 of
        the proof of
        \cite[Th.~4.1]{Reiter75}, that either there are $\quo{z}$ and $\quo{c}$ as above,
        or there are $f',g'\in V_{(i)}$ with $\veps\in\hat{\beta}_{(i)}(f')$ and $\veps\in\hat{\beta}_{(i)}(g')$ such
        that $s^{[\beta_{(i)}]}_{f',\veps,\veps}s^{[\beta_{(i)}]}_{g',\veps,\veps}(\quo{x})=\quo{y}$. In the former case we are done.
        In the latter case, let $f,g\in Ve_i$ be such that $\quo{f}=f'$ and $\quo{g}=g'$, and choose
        $c\in e^\sigma \hat{\beta}(f)e$ and $d\in e^\sigma\hat{\beta}(g)e$ with $\quo{c}=\quo{d}=\veps$.
        Then $c$ and $d$ are $(e^\sigma,e)$-invertible (Lemma~\ref{LM:ef-inv-mod-Jac}), and
        $\pi_i s_{f,e,c}^{[\beta]}s_{g,e,d}^{[\beta]}(x)=\pi_i y$. Replacing
        $x$ with $s_{f,e,c}^{[\beta]}s_{g,e,d}^{[\beta]}(x)$, we may assume $\quo{x}=\quo{y}$.
        Under this new assumption, we establish the existence of $\quo{z}$ and $\quo{c}$ as above.
        Indeed, this amounts to finding $\quo{z}\in V_{(i)}$
        with $h_{\beta_{(i)}}(\quo{x},\quo{z})=1$ and $\hat{\beta}_{(i)}(\quo{z})=\Lambda_{(i)}=\{0,1\}$.
        By \cite[Lm.~3.5]{Reiter75}, there is $\quo{z'}\in V_{(i)}$ with
        $h(\quo{x},\quo{z'})=1$. Condition (2c-$i$) and Lemma~\ref{LM:F-two-squared} now give there required $\quo{z}\in V_{(i)}$.
    \end{proof}

    \begin{proof}[Proof of Theorem~\ref{TH:Witt-I}]
        The proof is based on the inductive step of \cite[Th.~4.1]{Reiter75}
        with certain modifications.

\smallskip

        \Step{1}
        If $Q=0$ there is nothing to prove.
        Otherwise, by Lemma~\ref{LM:symmetric-modules}, $e_iA$ is isomorphic to a summand of $Q$ for some $i$.
        Since $\lAd{Q}:V\to Q^*$ is onto and $Q^*$ is a projective right $A$-module,
        we can write $V=W\oplus R$ such that $\lAd{Q}|_W$ is an isomorphism and $\lAd{Q}(R)=0$.
        Evidently, $(e_iA)^*\cong e_i^\sigma A\cong e_iA$ is also isomorphic to a summand $V'$ of $W$.
        If $(A_i,\sigma_i,u_i,\Lambda_i)$  is not split-orthogonal or $D_i\ncong \F_2\times\F_2$, we choose $V'$ arbitrarily.
        Otherwise, we claim that $V'$ can be chosen such that condition (2a), (2b) and (2c) are satisfied when $V$ is replaced
        with $V'\oplus R$.

        Indeed, assume $(A_i,\sigma_i,u_i,\Lambda_i)$  is  split-orthogonal and $D_i\ncong\F_2$.
        By condition (2a),  we have $[\beta_{(i)}]\neq[0]$.
        Choose some $x\in V_{(i)}$ with
        $\hat{\beta}_{(i)}(x)\neq\Lambda_{(i)}=0$, and write $x=w+r$ with $w\in W_{(i)}$, $r\in R_{(i)}$.
        If $w=0$, then $[\beta_{(i)}|_{R_{(i)}}]\neq [0]$, as required.
        Otherwise, $w\quo{A}$ is a summand of $\quo{W}$ (because $\quo{A}$ is semisimple) having $e_iA$ as projective
        cover. Thus, $W$ has a summand $V'\cong e_iA$ whose image in $\quo{W}$ is $wA$. That $V'$ satisfies
        $[\beta_{(i)}|_{V'_{(i)}\oplus R_{(i)}}]\neq [0]$, as required.

        The same strategy also works   when $D_i\cong \F_2$ by Lemma~\ref{LM:onto-is-preserved}(iii);
        start with $x\in V_{(i)}^\perp\cap V_{(i)}=(V_i^\perp\cap V_i)\veps_i$ such that $\hat{\beta}_{(i)}(x)\neq 0$.

        When $D_i\cong\F_2\times\F_2$, by Lemma~\ref{LM:F-two-squared} and conditions (2c) and (1a), there is $z\in
        V_{(i)}$ and $x\in Q_{(i)}$ such that $h_{\beta_{(i)}}(x,z)=\veps_i$ and $\hat{\beta}_{(i)}(z)=\Lambda_{(i)}=\{0,\veps_i\}$.
        Write $z=w+r$ with $w\in W_{(i)}$, $r\in R_{(i)}$, and choose $V'$  as  above.
        (We must have $w\quo{A}\cong \veps_i\quo{A}$. Otherwise, there is an idempotent $\delta\in A_{(i)}\cong \F_2\times\F_2$
        such that $\delta\neq\veps_i$ and $w\delta=w$, which implies $\veps_i=h_{\beta_{(i)}}(x,z)=h_{\beta_{(i)}}(w+r,z)=h_{\beta_{(i)}}(w,z)=
        \delta^\sigma h_{\beta_{(i)}}(w,z)\in\delta^\sigma A_{(i)}$, a contradiction.)
        Now, we have $z\in V'_{(i)}\oplus R_{(i)}$ such that $\hat{\beta}_{(i)}(z)=\{0,\veps_i\}$ and $h_{\beta_{(i)}}(x,z)=\veps_i$,
        so condition (2c) holds for $V'\oplus R$ by Lemma~\ref{LM:F-two-squared}.
        This settles our claim about the choice of $V'$.

\smallskip

        \Step{2}
        Write $W=V'\oplus V''$, and let $U'=\lAd{Q}(V')$ and $U''=\lAd{Q}(V'')$.
        Then $Q^*=U'\oplus U'$, so by Lemma~\ref{LM:dual-decomp},
        we have a decomposition $Q=Q'\oplus Q''$ such that $\lAd{Q}(V')=\{f\in Q^*\suchthat f(Q'')=0\}=Q'^*$ and
        $\lAd{Q}(V'')=\{f\in Q^*\suchthat f(Q')=0\}=Q''^*$.
        Now, $Q'^*\cong V'\cong e_iA$, hence $Q'\cong Q'^{**}\cong (e_iA)^*\cong e_iA$,
        so $Q'\cong Q'^*$. Thus,
        $Q'\oplus Q''^*\cong (Q'\oplus Q'')^*=Q^*\cong Q\cong Q'\oplus Q''$, which
        implies $Q''^*\cong Q''$ (see the discussion at the end of~\ref{subsection:semiperfect}).
        Thus, we may apply induction to $(Q'',[\beta|_{Q''}])$ and $\psi|_{Q''}$.
        This yields  $\vphi\in O(P,[\beta])$,
        a product of reflections taken with respect to elements of $V$,
        such that $\vphi|_{Q''}=\psi|_{Q''}$.

\smallskip

        \Step{3}
        We now claim that conditions (1a)--(1c) and (2a)--(2c)
        hold for the isometry $\vphi^{-1}\psi|_Q:Q'\to \vphi^{-1}\psi(Q')$ and the module
        $V'\oplus R$ (in place of $V$).
        Indeed, (2a)--(2c) hold by our choice of $V'$ (cf.\ Step~1), and (1c) was verified in Step~2.

        Next, notice that by construction, an element $z\in V$ lives in $V'\oplus R$
        if and only if $\lAd{Q}(z)\in Q'^*$, i.e.\ $0=\lAd{Q}(z)(Q'')=h(z,Q'')$.
        Now, for all $y\in Q''$, we have
        \begin{eqnarray*}
        h(\vphi^{-1}\psi x-x,y)&=&h(\vphi^{-1}\psi x,y)-h(x,y)=h(\psi x,\vphi y)-h(x,y)\\
        &=&h(\psi x,\psi y)-h(x,y)=h(x,y)-h(x,y)=0\ .
        \end{eqnarray*}
        This implies that $\vphi^{-1}\psi x-x\in V'\oplus R$ for all $x\in Q'$, so (1b) holds.

        Finally, we have $\lAd{Q'}(V'\oplus R)=Q'^*$
        by construction,  so to show (1a) amounts to showing $\lAd{\vphi^{-1}\psi Q'}(V'\oplus R)=(\vphi^{-1}\psi Q')^*$.
        Let $f\in (\vphi^{-1}\psi Q')^*$. We extend $f$
        to $\vphi^{-1}S=\vphi^{-1}\psi Q=\vphi^{-1}\psi Q'\oplus \vphi^{-1}\psi Q''=\vphi^{-1}\psi Q'\oplus Q''$
        by setting it to be $0$ on $Q''$.
        View $f\vphi^{-1}$ as an element of $S^*$. By
        (1a) (for $Q,S,V$), there is $x\in V$
        with $\lAd{S}(x)=f\vphi^{-1}$.
        Now, for all $y\in \vphi^{-1}\psi Q$, we have
        $h(\vphi^{-1}x, y)=h(x,\vphi y)=\lAd{S}(x)(\vphi y)=f\vphi^{-1}\vphi y=fy$,
        so $\lAd{\vphi^{-1}\psi Q'}(\vphi^{-1}x)=f$. Since $\vphi$ is  a product of reflections
        taken with respect to element of $V$, $\vphi(V)=V$, hence $\vphi^{-1}x\in V$. In addition,
        when $y\in Q''$, we have $h(\vphi^{-1}x,y)=fy=0$, so $\vphi^{-1}x\in V'\oplus R$.
        This shows that $f\in \lAd{\vphi^{-1}\psi Q'}(V'\oplus R)$.

\smallskip

        \Step{4} To finish,
        we apply Lemma~\ref{LM:main-lemma} to $\vphi^{-1}\psi|_Q:Q'\to \vphi^{-1}\psi(Q')$ with
        $V'\oplus R$ in place of $V$ to get a product $\eta$ of $e_i$-reflections, taken with respect to elements
        of $V'\oplus R$, such that $\eta|_{Q'}=\vphi^{-1}\psi$.
        Since $\lAd{Q''}(V'\oplus R)=0$, reflections taken with respect to elements
        of $V'\oplus R$ fix $Q''$. Thus, $\eta|_{Q''}=\id_{Q''}$, so $\eta|_Q=\vphi^{-1}\psi$.
        This means $\psi=\vphi\eta|_Q$ and we are done.
    \end{proof}

    \begin{remark}\label{RM:looser-conds}
        Theorem~\ref{TH:Witt-I} actually holds when (2c) is replaced with the milder conditions:
        \begin{enumerate}
            \item[(2c$'$)] If $D_i\cong\F_2\times\F_2$ and $n_i=1$, then  there is no $z\in V_{(i)}=V_i$ satisfying
            $h_{\beta_{i}}(z,z)=\veps_i=1_{A_i}$ and $[\beta_{i}|_{z^\perp\cap V_{i}}]=[0]$.
            \item[(2c$''$)] If $D_i\cong\F_2\times\F_2$ and $n_i>1$, then $Q_i\ncong \veps_iA_i$.
        \end{enumerate}
        The idea is to extend Lemma~\ref{LM:main-lemma} to the case
        $D_i\cong\F_2\times\F_2$, $n_i>1$, and $Q\cong e'_iA$, where
        $e'_i$ is a lifting of the idempotent $\veps'_i=e_{11}+e_{22}\in A_i$.
        The proof requires
        using both $e_i$-reflections (to get $\quo{x}=\quo{y}$) and $e'_i$-reflections (to find $\quo{z}$ and $\quo{c}$; use
        \cite[Lm.~3.6]{Reiter75}).
        The inductive step of Theorem~\ref{TH:Witt-I} should also be modified. In case
        $D_i\cong\F_2\times\F_2$ and $n_i>1$, one chooses $V'$ isomorphic to $e'_iA$ rather than $e_iA$.
        If this results in  $Q''_i\cong \veps_iA_i$,
        then one should split $Q'\cong e'_iA$ into two copies of $e_iA$ and apply the induction hypothesis to the direct sum
        of $Q''$ and one of these copies.
        The full and detailed argument would overload the document while not benefiting any result
        beside Theorem~\ref{TH:Witt-I} and Corollary~\ref{CR:main-I} below, so we chose to omit it.
    \end{remark}

    If we do not require in Theorem~\ref{TH:Witt-I} that the extension of $\psi$  will
    be  a product of quasi-reflections, then conditions (1c), (2a), (2b), (2c) can be dropped.

    \begin{thm}\label{TH:Witt-II}
        Let $(P,[\beta])$ be a quadratic space, let $Q,S,V$ be summands of
        $P$, and let $\psi:(Q,[\beta|_Q])\to (S,[\beta|_S])$ be an isometry.
        Assume that
        \begin{enumerate}
            \item[(1a)] $\lAd{Q}(V)=Q^*$ and $\lAd{S}(V)=S^*$, and
            \item[(1b)] $\psi x-x\in V$ for all $x\in V$.
        \end{enumerate}
        Then $\psi$ extends to an isometry
        $\vphi\in O(P,[\beta])$ such that $\vphi x-x\in V$ for all $x\in P$.
    \end{thm}

    \begin{proof}
        The proof is essentially the same as  the proof of \cite[Th.~6.2]{Reiter75}.
        However, since there are some differences in the conditions to be checked, we recall the argument.

        Assume first that (1c) holds. Let $T$ be a free $A$-module
        with basis $\{z,w\}$. We choose $a\in A$ such that:
        \begin{itemize}
        \item
        $a_i:=\pi_i a=0$ if $D_i\ncong \F_2\times\F_2$, or $D\cong \F_2\times\F_2$ and $\lAd{[\beta_{(i)}],V_{(i)}}(V_{(i)})\ncong A_{(i)}$.
        \item
        $a_i:=\pi_i a\in A_{(i)}\setminus\{0,\veps_i\}$ if $D_i\cong \F_2\times\F_2$ and $\lAd{[\beta_{(i)}],V_{(i)}}(V_{(i)})\cong A_{(i)}$.
        \end{itemize}
        We define a sesquilinear form $\gamma:T\times T\to A$ by linearly extending
        \begin{align*}
            \gamma(z,z)=&0 & \gamma(z,w)=&1 \\
            \gamma(w,z)=&0 & \gamma(w,w)=&a
        \end{align*}
        Let $(P',[\beta'])=(P,[\beta])\perp (T,[\gamma])$, $Q'=Q\oplus zA$, $S'=S\oplus zA$, $V'=V\oplus (z+w)A$
        and define $\psi':Q'\to S'$ by $\psi'(x\oplus zc)=\psi x\oplus zc$ ($c\in A$). It
        is easy to check that $\psi'$ is an isometry from $[\beta'|_{Q'}]$ to $[\beta'|_{S'}]$.
        Furthermore, the conditions of Theorem~\ref{TH:Witt-I} hold
        for $(P',[\beta']),Q',S',V',\psi'$:
        \begin{enumerate}
        \item[(1a)] holds since $h_{\beta'}(0\oplus (z+w),0\oplus z)=u\in\units{A}$,
        \item[(1b)] is straightforward,
        \item[(1c)] holds since $Q'^*\cong Q^*\oplus A^*\cong Q\oplus A\cong Q'$,
        \item[(2a)] holds since $\hat{\gamma}_{(i)}(\pi_i(ze_i+we_i))=\veps_i+a_i+\Lambda_{(i)}$,
        \item[(2b)] holds since $\pi_i(0\oplus (ze_i+we_i))\in V'_{(i)}\cap V'^\perp_{(i)}$ when $D_i\cong \F_2$ (because
        then $h_{{\gamma}_i}(\pi_i(ze_i+we_i),\pi_i(ze_i+we_i))=2\veps_i=0$), and
        \item[(2c)] holds by Lemma~\ref{LM:F-two-squared}, because if $D_i\cong\F_2\times\F_2$,
        then
        $\lAd{[\beta'_{(i)}],V'_{(i)}}(V'_{(i)})\ncong A_{(i)}$.
        Indeed,  we have $\lAd{[\beta'_{(i)}],V'_{(i)}}(V'_{(i)})\cong\lAd{[\beta_{(i)}],V_{(i)}}(V_{(i)})\oplus (a_i+a_i^\sigma)A$
        because $h_{{\gamma}_{(i)}}(\pi_i(ze_i+we_i),\pi_i(ze_i+we_i))=a_i+a_i^\sigma$.
        \end{enumerate}
        Thus, $\psi'$ extends to $\vphi'\in O(P',[\beta'])$ with $\vphi' x-x\in V'$ for all $x\in P'$.
        We finish by showing that $\vphi'(P\oplus 0)=P\oplus 0$. Indeed, for $x\oplus 0\in P\oplus 0$,
        we can write $\vphi'(x\oplus 0)=y\oplus (z+w)c$ for some $y\in P$, $c\in A$.
        This implies $0=h_{\beta'}(0\oplus z,x\oplus 0)= h_{\beta'}(\vphi'(0\oplus z),\vphi'(x\oplus 0))
        =h_{\beta'}(0\oplus z,y\oplus (z+w)c)=c$,
        so $c=0$ and $\psi'(x\oplus 0)\in P\oplus 0$.

        When (1c) does not hold, choose $U\in\rproj{A}$ such that $(Q\oplus U)^*\cong Q\oplus U$.
        Define a sesquilinear form $\gamma$ on $U\oplus U^*$ by $\gamma(x\oplus f,y\oplus g)=fy$
        ($x,y\in U$, $f,g\in U^*$;
        $(U\oplus U^*,[\gamma])$ is the \emph{hyperbolic quadratic space} associated with $U$).
        Now, let $(P',[\beta'])=(P,[\beta])\oplus (U\oplus U^*,[\gamma])$, $Q'=Q\oplus U\oplus 0$, $S'=S\oplus U\oplus 0$, $V'=V\oplus0\oplus U^*$,
        and define $\psi':Q'\to S'$ by $\psi'(x\oplus z\oplus 0)=\psi x\oplus z\oplus 0$ ($x\in Q$, $z\in U$).
        Conditions (1a), (1b), (1c) are easily
        seen to hold for $(P',[\beta']),Q',S',V',\psi'$, so by what we have shown above, $\psi'$ extends to an isometry
        $\vphi'\in O(P',[\beta'])$ with $\vphi' x-x\in V'$. Similar arguing implies that $\vphi'$ maps
        $P$ into itself.
    \end{proof}

    The following corollaries
    can be viewed as analogues of Witt's Extension Theorem
    and Witt's Cancellation Theorem.

    \begin{cor}\label{CR:main-I}
        Let $(P,[\beta])$ be a quadratic space and let $Q,S$ be summands
        of $P$. Assume that  $(P,[\beta])$ or $(Q,[\beta|_Q])$ is unimodular.
        Then any isometry of $\psi:(Q,[\beta|_Q])\to(S,[\beta|_S])$ extends to
        an isometry $\vphi\in O(P,[\beta])$.
    \end{cor}

    \begin{proof}
        Take $V=P$ in Theorem~\ref{TH:Witt-II}. Condition (1b) is clear, and condition
        (1a) follows from the unimodularity assumption.
    \end{proof}

    \begin{cor}\label{CR:main-II}
        Let $(P,[\beta])$, $(Q_1,[\gamma_1])$, $(Q_2,[\gamma_2])$
        be quadratic spaces such that $(P,[\beta])$ is unimodular
        and $(P,[\beta])\perp(Q_1,[\gamma_1])\cong (P,[\beta])\perp (Q_2,[\gamma_2])$.
        Then $(Q_1,[\gamma_1])\cong (Q_2,[\gamma_2])$.
    \end{cor}

    \begin{proof}
        Identify $(Z,[\zeta]):=(P,[\beta])\perp(Q_1,[\gamma_1])$ with $(P,[\beta])\perp (Q_2,[\gamma_2])$.
        Corollary~\ref{CR:main-I} implies that any isometry between the two copies of $(P,[\beta])$ in $(Z,[\zeta])$ extends
        to an isometry of $(Z,[\zeta])$. This isometry must map $(Q_1,[\gamma_1])$ isometrically onto $(Q_2,[\gamma_2])$.
    \end{proof}

    \begin{cor}\label{CR:main-III}
        Let $(P,[\beta])$ be a unimodular quadratic space. Assume that
        \begin{enumerate}
            \item[(1)] if $(A_i,\sigma_i,u_i,\Lambda_i)$ is  split-orthogonal, then $D_i\ncong\F_2$, and
            \item[(2)] if $D_i\cong\F_2\times\F_2$, then $P_i\ncong \veps_i A_i$.
        \end{enumerate}
        Then $O(P,[\beta])$ is generated by quasi-reflections.
    \end{cor}

    \begin{proof}
        It is enough to show that the conditions of Theorem~\ref{TH:Witt-I}
        hold when we take $V,Q,S$ to be $P$. Indeed, conditions (1a) and (1b) hold as in the proof of Corollary~\ref{CR:main-I}.
        Condition (1c) holds because $L_{[\beta]}:P\to P^*$ is an isomorphism.
        To see (2a), observe that
        by Lemma~\ref{LM:onto-is-preserved}(ii),
        $\lAd{P_{(i)}}:P_{(i)}\to P_{(i)}^*$ is an isomorphism, and hence either $P_{i}=0$ or $[\beta_{i}]\neq 0$.
        Condition (2b) follows from assumption (1). Finally, assumption (2) and the unimodularity of $[\beta]$ imply,
        $\lAd{V_{(i)}}(V_{(i)})=\lAd{P_{(i)}}(P_{(i)})=P_{(i)}^*\cong P_{(i)}\ncong A_{(i)}$, so condition (2c) holds
        by Lemma~\ref{LM:F-two-squared}.
    \end{proof}

    \begin{remark}\label{RM:comparison}
        Theorems~\ref{TH:Witt-I} and~\ref{TH:Witt-II} should be compared with Theorems 6.1 and~6.2 of \cite{Reiter75}
        (and also Theorems~4.1, and~5.1 in \cite{Reiter75}).

        Theorem~6.1 in \cite{Reiter75}
        is similar to Theorem~\ref{TH:Witt-I}, but it  applies to the broader class
        of semilocal rings, and it derives the stronger conclusion that  $\vphi$ is is a product
        of \emph{reflections} (rather than quasi-reflections). However, it assumes the stronger assumptions
        \begin{enumerate}
            \item[(1c$'$)] $Q$ and $S$ are free,
            \item[(2a$'$)] if $(A_i,\sigma_i,u_i,\Lambda_i)$ is split-orthogonal, then $[\beta_i|_{Q_i^\perp\cap V_i}]\neq [0]$,
        \end{enumerate}
        in place of (1c) and (2a), and condition (2c) is replaced
        with condition (2c$'$) of Remark~\ref{RM:looser-conds} (note that (1c$'$) and (2c$'$) imply (2c$''$)).
        In  section~\ref{section:gen-by-ref}, we shall give a precise description of which isometries
        are products of reflection (in case $A$ is semiperfect) from which it will follow that condition (2a$'$) cannot
        be completely removed in general.

        Theorem~6.1 in \cite{Reiter75} resembles Theorem~\ref{TH:Witt-II}. It applies to semilocal unitary rings, but it
        assumes condition (2a$'$) in addition to (1a) and (1b).
        This additional assumption does not allow one to derive Corollaries~\ref{CR:main-I} and~\ref{CR:main-II} from
        \cite[Th.~6.2]{Reiter75}, and in fact, these corollaries are false for semilocal rings, as demonstrated by Keller
        \cite[\S2]{Keller88}. Nevertheless, \cite[Th.~6.2]{Reiter75} still implies
        Corollary~\ref{CR:main-II} in case $[(\gamma_1)_i]\neq 0$ whenever $(A_i,\sigma_i,u_i,\Lambda_i)$
        is split-orthogonal.
        See also \cite[Th.~3.4]{Keller88} for different proof and strengthening of this statement.
    \end{remark}

    \begin{remark}
        Condition (2b) in Theorem~\ref{TH:Witt-I} (and also \cite[Th.~6.1]{Reiter75})
        cannot be completely removed, even when $(P,[\beta])$ is unimodular.
        Indeed, over $\F_2$, there is a unimodular quadratic form of dimension $4$
        whose isometry group is not generated by quasi-reflections (which are just reflections
        in this case). However, up to isomorphism, this is the only
        exception over $\F_2$ (see \cite[Cr.~11.42]{Taylor92GeometryOfClassicalGroups}, for instance).
        We do not know if condition (2c) can be removed in general.
    \end{remark}

\subsection{Applications}

    We now use the previous results to derive several consequences to hermitian categories and systems of sesquilinear forms.
    Hermitian categories are a purely categorical framework to work with quadratic
    forms. We refer the reader to \cite[\S1]{QuSchSch79}, \cite[Ch.~7]{SchQuadraticAndHermitianForms} or \cite[Ch.~II]{Kn91}
    for the relevant definitions.

\smallskip

    Let $(\catC,*,\omega)$ be a hermitian category with form parameter $(\veps,\Lambda)$.
    In \cite[\S3.4]{QuSchSch79},
    Quebbemann, Scharlau and Schulte prove that unimodular $(\veps,\Lambda)$-quadratic forms over
    $\catC$ cancel from orthogonal sums, provided the following assumptions hold:
    \begin{enumerate}
        \item[(A)] All idempotent morphisms in $\catC$ split.
        \item[(B)] Every object of $\catC$ is a direct sum of finitely many objects
        with local endomorphism ring (or, alternatively, the endomorphism ring of every object in $\catC$
        is semiperfect).
        \item[(C)] For every $M\in\catC$, $E:=\End_{\catC}(M)$ is complete in the $\Jac(E)$-adic topology.
    \end{enumerate}
    After proving this consequence, the authors comment that condition (C) is in fact unnecessary since
    one can apply   \emph{transfer} (see \cite[Pr.~2.4]{QuSchSch79}) to translate
    the problem to cancellation of quadratic forms over semiperfect unitary rings, and then apply
    Reiter's Theorem~6.2 in \cite{Reiter75}. This also implies that condition (A) is unnecessary.
    However, the authors seem unaware that Reiter's Theorem
    does not imply cancellation in general (see Remark~\ref{RM:comparison}), and hence this argument
    implies cancellation only in certain
    cases (e.g. when all three $(\veps,\Lambda)$-quadratic spaces
    involved are unimodular and their base objects are \emph{progenerators} of $\catC$).

    Nevertheless, by replacing \cite[Th.~6.2]{Reiter75} with Theorem~\ref{TH:Witt-II}
    (or Corollary~\ref{CR:main-II}), we see that condition (C) can indeed be  dropped.
    We have therefore obtained:

    \begin{cor}\label{CR:cancellation-herm-cats}
        Assume condition  (B) is satisfied. Then unimodular $(\veps,\Lambda)$-quadratic forms over
        $\catC$ cancel from orthogonal sums.
    \end{cor}

    By the same methods (i.e.\ via transfer in the sense of \cite[Pr.~2.4]{QuSchSch79}), Corollary~\ref{CR:main-I}
    implies:

    \begin{cor}
        Assume condition (B) is satisfied, let $(P,[\beta])$ be
        a unimodular $(\veps,\Lambda)$-quadratic space, and let $Q,S$ be summands
        of $P$. Then any isometry $\psi:(Q,[\beta|_Q])\to(S,[\beta|_S])$ extends to
        an isometry $\vphi\in O(P,[\beta])$.
    \end{cor}

    We now combine Corollary~\ref{CR:cancellation-herm-cats} with  results from  \cite{BayerFain96}, \cite{BayerMold12} and \cite{BayFiMol13}
    to obtain cancellation results for systems of (not-necessarily unimodular) sesquilinear forms.
    Henceforth, let $R$ be a commutative  ring, and let $\catC$ be an $R$-category
    equipped with $R$-linear hermitian structures $\{*_i,\omega_i\}_{i\in I}$; see \cite[\S2.4, \S4]{BayFiMol13} for
    the definitions.

    \begin{cor}\label{CR:systems}
        Assume that $R$ is local and henselian, $2\in\units{R}$, and
        at least one of the following holds:
        \begin{enumerate}
            \item[(1)] $R$ is noetherian and
            $\End_{\catC}(M)$ is finitely generated as an $R$-module for all $M\in\catC$.
            \item[(2)] $R$ is a valuation ring, and $\End_{\catC}(M)$ is $R$-torsion-free
            and has finite rank
            for all $M\in\catC$.
        \end{enumerate}
        Then systems of sesquilinear forms over $(\catC,\{*_i,\omega_i\}_{i\in I})$ cancel from orthogonal sums.
    \end{cor}

    \begin{proof}
        This is similar to the proof of \cite[Th.~5.2]{BayFiMol13}.
        By \cite[Th.~4.1]{BayFiMol13}, the category of sesquilinear forms over
        $(\catC,\{*_i,\omega_i\}_{i\in I})$ is equivalent to the category of unimodular
        $1$-hermitian forms over another hermitian category $(\catD,*,\omega)$, where $\catD$ is the category $\TDA[I]{\catC}$
        constructed in \cite[\S4]{BayFiMol13}.
        It therefore enough to prove that unimodular $1$-hermitian forms over $(\catD,*,\omega)$
        cancel from orthogonal sums. Note that since $2\in \units{R}$, there is only one form parameter
        $(1,\Lambda)$ for $(\catD,*,\omega)$, and $1$-hermitian forms and $(1,\Lambda)$-quadratic forms are essentially
        the same.
        By construction, the endomorphism ring  of an object  $Z\in\catD$ is a sub-$R$-algebra
        of  $\End_{\catC}(M)\times\End_{\catC}(N)^\op$ for some $M,N\in \catC$, so by
        Corollary~\ref{CR:cancellation-herm-cats}, it is enough to check that such rings are semiperfect.
        This is indeed the case by Proposition~\ref{PR:semiperfect-examples}.
    \end{proof}

    As a special case of Corollary~\ref{CR:systems}, we get:

    \begin{cor}\label{CR:systems-II}
        Let $A$ be an $R$-algebra and let $\{\sigma_i\}_{i\in I}$ be a system of $R$-involutions
        on $A$.
        Assume that is $R$ local and henselian, $2\in\units{R}$, and at least one of the following holds:
        \begin{enumerate}
            \item[(1)] $R$ is noetherian and
            $A$ is finitely generated as an $R$-module.
            \item[(2)] $R$ is a valuation ring, and $A$ is $R$-torsion-free
            of finite rank.
        \end{enumerate}
        Then systems of sesquilinear forms over $(A,\{\sigma_i\})$ (that are defined on f.g.\ projective $A$-modules)
        cancel from orthogonal sums.
    \end{cor}

    Finally, we  use  \cite{BayerFain96}
    to show that  \emph{non-unimodular} hermitian forms over \emph{non-henselian} valuation
    rings cancel from orthogonal sums. Here, the prefix ``non-'' stands for ``not-necessarily'' and
    not for absolute negation.

    \begin{cor}\label{CR:non-unimod}
        Let $(R,\sigma)$ be an arbitrary valuation ring with involution, and let $u\in R$ be such that
        $u^\sigma u=1$. Suppose that
        there exists $r\in R$ with $r+r^{\sigma}u\in\units{R}$ (e.g.\ if $1+u\in \units{R}$).
        Then non-unimodular hermitian forms over $(R,\sigma)$ cancel from orthogonal sums.
    \end{cor}

    \begin{proof}
        Write $v=r+r^\sigma u$. Then $v\in \units{R}$ and $v^\sigma u=v$.
        By conjugating with $v^{-1}$ (see Proposition~\ref{PR:conjugation-of-unitary-rings}) and replacing
        $r$ with $rv^{-1}$, we may
        assume that $u=1$ and $r+r^\sigma=1$.

        By \cite[Th.~4.1]{BayerFain96}, the category
        of \emph{arbitrary} $1$-hermitian forms over $(R,\sigma)$ is equivalent
        to the category of \emph{unimodular} $1$-hermitian forms over a hermitian category $(\catD,*,\omega)$ (see \cite[\S3]{BayerFain96} for its definition).
        The category $\catD$ is the category of \emph{morphisms in $\rproj{R}$}, denoted $\Mor(\rproj{R})$. (One can also
        use \cite[Th.~3.2]{BayFiMol13} to get essentially the same result.)
        Notice that the $\Hom$-sets of $\catD$ are $R$-modules and it not hard to check that $(f\cdot a)^*=f\cdot a^\sigma$
        for every morphism $f$ in $\catD$ and $a\in R$.
        Since $r+r^\sigma =1$,
        the map $(Q,h)\mapsto (Q,[rh])$ taking $1$-hermitian forms over $\catD$
        to $(1,\Lambda^{\min})$-quadratic spaces over $\catD$ is an isomorphism of categories
        (and  there is only one form parameter $(1,\Lambda)$ for $(\catD,*,\omega)$).
        Now, by Corollary~\ref{CR:cancellation-herm-cats}, it is enough to prove
        that every object of $\Mor(\rproj{R})$ is the direct sum of objects with local endomorphism rings.

        Recall that an object of $\Mor(\rproj{R})$ consists of a triple $(P,f,P')$ such
        that $P,P'\in \rproj{R}$ and $f\in\Hom_R(P,P')$. The endomorphism ring of $(P,f,P')$
        is the set of pairs $(g,g')\in\End_R(P)\times\End_R(P')$ such that $g'f=fg$.
        Using the fact that $R$ is a valuation ring, one can show that
        every $(P,f,P')$ is the direct sum of objects of the form
        $(R,0,0)$, $(0,0,R)$ and $(R,t,R)$ with $t\neq 0$: The proof essentially boils
        down to showing that for every matrix $f=(f_{ij})\in\nMat{R}{n\times m}$
        there are invertible matrices $w\in \nMat{R}{n}$ and $w'\in\nMat{R}{m}$ such
        that $(a_{ij}):= wfw'$ is diagonal (meaning that $a_{ij}=0$ for $i\neq j$),
        and this is
        well-known. The endomorphism rings of $(R,0,0)$, $(0,0,R)$ and $(R,t,R)$ are easily
        seen to be $R$, which is local, so we are done.
    \end{proof}

\section{Generation by Reflections}
\label{section:gen-by-ref}

    Let $(P,[\beta])$ be a unimodular quadratic
    space defined over a semiperfect unitary ring. Denote
    by $O'(P,[\beta])$ the subgroup of $O(P,[\beta])$ generated by
    reflections. In this section, we  describe $O'(P,[\beta])$ and show that, apart from
    an obvious exception, it is a subgroup of finite index in $O(P,[\beta])$, and that index
    can be determined in terms of $A$ and $P$.

\subsection{Dickson's Invariant}
\label{subsection:dickson}

    Let $(A,\sigma,u,\Lambda)$ be an orthogonal simple artinian unitary ring
    (see \ref{subsection:orthogonality}), let $K=\Cent(A)$, and let $(P,[\beta])$
    be a unimodular quadratic space.
    We now recall \emph{Dickson's Invariant} (also called \emph{pseudodeterminant} or \emph{quasideterminant}),
    which is a group homomorphism
    \[\Delta=\Delta_{[\beta]}:O(P,[\beta])\to \Z/2\Z\ . \]
    This homomorphism will
    be used in describing what is $O'(P,[\beta])$.
    The map $\Delta$ is related to the reduced norm map on $E:=\End_A(P)$
    via
    \begin{equation}\label{EQ:Dickson-determinant-connection}
    \rdnorm[E/K](\psi)=(-1)^{\Delta(\psi)}\qquad\forall \psi\in O(P,[\beta]) \ ,
    \end{equation}
    so one can use $\rdnorm[E/K]$ to define $\Delta$ in case $\Char K\neq 2$.

    Unfortunately, there seems to be no reference defining Dickson's invariant directly for
    quadratic forms over noncommutative central simple algebras (but see~\cite[p.~160]{Taylor92GeometryOfClassicalGroups}
    for the case $A=K$, and \cite{Ba74}, \cite[\S{}IV.5]{Kn91} for arbitrary commutative rings).
    We  therefore satisfy with giving  an ad-hoc definition that suits our needs,
    but may seem unnatural. In the end, we shall briefly explain how to obtain a more intrinsic
    definition
    by using the Dickson invariant of \emph{quadratic pairs} introduced in \cite[\S12C]{InvBook}.

\medskip

    We  start by recalling  Dickson's invariant in the case $A=K$ (the orthogonality then forces
    $\sigma=\id_K$, $u=1$ and $\Lambda=0$). In this case,
    $(P,[\beta])$
    is just a quadratic space in the classical sense. Dickson's invariant can then be defined by
    \begin{align*}
        \Delta(\psi)= \dim_K\im(1-\psi)+ 2\Z \qquad\forall\, \psi\in O(P,[\beta])\ .
    \end{align*}
    The map $\Delta$ is indeed a group homomorphism (\cite[Th.~11.43]{Taylor92GeometryOfClassicalGroups}),
    and moreover, it is a morphism of  algebraic groups over $K$ (\cite{Ba74}).
    See also \cite[Df.~12.12]{InvBook} and \cite[p.~129]{Grove02ClassicalGroups}
    for alternative definitions in characteristic $2$. (The equivalence of the definitions, and the identity
    \eqref{EQ:Dickson-determinant-connection}, can be verified using the references specified.
    In particular, one should only verify that the definitions coincide on reflections
    by \cite[Th.~11.39]{Taylor92GeometryOfClassicalGroups}.)

    Assume now that $A$ is arbitrary and let $E=\End_A(P)$. Then $E$ is a central simple $K$-algebra.
    We define $\Delta$ by
    \[
    \Delta(\psi)=\frac{\dim_K(1-\psi)E}{\deg E} +2\Z \qquad \forall\, \phi\in O(P,[\beta])
    \]
    It is easy to see that this definition agrees with the  definition in the case $A=K$. However,
    we have to check that $\Delta$ is indeed a homomorphism of  groups.

    \begin{prp}\label{PR:Delta-is-a-homomorphism}
        If $A$ is non-split, then $\Delta\equiv 0$. Otherwise, $\Delta$ is a
        surjective group homomorphism, and furthermore,
        it is a homomorphism of algebraic groups defined over $K$.
    \end{prp}

    \begin{proof}
        Assume $A$ is non-split. If $\Delta(\psi)=1+2\Z$, then $\deg E$ must be odd (because
        it divides $\dim_K(1-\psi)E$). On the other hand, $E$ is Brauer equivalent
        to $A$, which has an involution of the first
        kind (Proposition~\ref{PR:conjugation-to-standard-form}), and hence $\ind E=\ind A$ is a power of two (\cite[Cr.~2.6]{InvBook}).
        Thus, $\ind A=1$, i.e.\ $A$ is split. (Compare with \cite[Lm.~2.6.1]{Knes69GaloisCohomology}.)

        Assume $A$ is split.
        Observe that  the definition of $\Delta$ depends only on the isomorphism class of the endomorphism
        ring $E=\End_A(P)$, which does not change under conjugation and $e$-transfer (see Propositions~\ref{PR:conjugation-of-unitary-rings}
        and~\ref{PR:transfer-special-case}). These
        transitions also do not affect split-orthogonality by Proposition~\ref{PR:orth-not-affected-by-conj},
        so we may freely apply them.
        Now, by conjugation,
        we may  assume $(A,\sigma,u,\Lambda)$ is standard (Proposition~\ref{PR:conjugation-to-standard-form}),
        i.e.\ $u=1$ and $\sigma$ is the matrix transpose involution. Let $e$ be the matrix unit $e_{11}$.
        By applying $e$-transfer, we may assume $A=K$.
        But in this case it is known that $\Delta$ is a surjective group homomorphism,
        and also a homomorphism
        of algebraic groups defined over $K$.
    \end{proof}

    We now compute the Dickson invariant of reflections.

    \begin{prp}\label{PR:Dickson-inv-of-reflections}
        Let $(A,\sigma,u,\Lambda)$ be an orthogonal
        simple artinian unitary ring, let $e\in A$ be an idempotent,
        and let $(P,[\beta])$ be a unimodular quadratic space over $(A,\sigma,u,\Lambda)$
        with $P\neq 0$.
        Then for every $e$-reflection $s\in O(P,[\beta])$ we have
        \[
        \Delta(s)=\deg eAe+2\Z\ .
        \]
    \end{prp}

    \begin{proof}
        We assume $e\neq 0$; the proposition becomes trivial otherwise.
        Also, if $A$ is non-split, then $\ind eAe=\ind A$ is even, so we are done by Proposition~\ref{PR:Delta-is-a-homomorphism}.

        Assume $A$ is split, i.e.\ $A=\nMat{K}{n}$ where $K$ is a field. By conjugation, we may
        assume that $\sigma$ is the transpose involution and $u=1$
        (Proposition~\ref{PR:conjugation-to-standard-form}, Remark~\ref{RM:conjugation-does-not-affect-reflections}(i)).
        Let $\{e_{ij}\}$ be the standard matrix units of $A$.
        Then $e$ is equivalent to $\sum_{i=1}^re_{ii}$ for some $r$. By Lemma~\ref{LM:reflections-manipulations},
        me replace $e$ with $\sum_{i=1}^re_{ii}$ to assume $e^\sigma=e$.
        Now,
        by applying $e$-transfer, we may
        assume $e=1$ (Proposition~\ref{PR:transfer-special-case}, Remark~\ref{RM:conjugation-does-not-affect-reflections}(ii)).
        We further tensor $(A,\sigma,u,\Lambda)$ and $(P,[\beta])$ with an algebraic closure $F$ of $K$, to assume
        $K$ is algebraically closed. (Namely, we replace $(A,\sigma,u,\Lambda)$ with
        $(A\otimes_KF,\sigma\otimes_K\id_{F},u\otimes 1,\Lambda\otimes_KF)$,
        and $(P,[\beta])$ with $(P\otimes_KF,[\beta_{F}])$, where
        $\beta_{F}$ is defined by $\beta_{F}(x\otimes a,x'\otimes a')=\beta(x,x')\otimes aa'$
        for all $x,x'\in P$, $a,a'\in A$. The details are left to the reader.)

        Let $S$ be the set of pairs $(y,\lambda)\in P\times\Lambda^{\min}$ satisfying
        $\beta(y,y)+\lambda\in\nGL{K}{n}$. Then $S$ is open in the Zaritzki topology of $P\times \Lambda$
        (viewed as an affine space over $K$), and hence connected.
        Since  $\Delta:O(P,[\beta])\to \Z/2\Z$ is a morphism of algebraic groups
        (Proposition~\ref{PR:Delta-is-a-homomorphism}),
        the map
        \[(y,\lambda)\mapsto \Delta(s_{y,1,\beta(y,y)+\lambda})~:~S \to \Z/2\Z\]
        is continuous in the Zaritzki topology.
        Therefore, since $S$ is connected,   it is  enough to prove that
        $\Delta(s)=\deg A+2\Z=n+2\Z$ for \emph{some} reflection $s$.

        Let $\{e_{ij}\}_{i,j}$ be the standard matrix units of $A=\nMat{K}{n}$.
        By Lemma~\ref{LM:reflections-manipulations}(i), every $e_{11}$-reflection is an
        $e_{ii}$-reflection for any $i$, so by Lemma~\ref{LM:reflections-manipulations}(ii),
        the product of $n$ $e_{11}$-reflections is a reflection.
        Since $\Delta$ is a group homomorphism, it is
        therefore enough to prove that there is an $e_{11}$-reflection $s$ with $\Delta(s)=1+2\Z$.
        Now,
        note that $e_{11}Ae_{11}\cong K$.
        Applying $e_{11}$-transfer (together with Remark~\ref{RM:conjugation-does-not-affect-reflections}(ii)), we are reduced to
        show that when $A=K$, there exists a reflection
        $s$ with $\Delta(s)=1+2\Z$.
        But this well-established; see \cite[Ex.~12.13]{InvBook}, for instance.
    \end{proof}

    \begin{remark}\label{RM:alt-def-of-Dickson}
        Another way to define $\Delta_{[\beta]}$ when $A$ is an arbitrary
        central simple $K$-algebra is by relating $(P,[\beta])$ with a \emph{quadratic pair}
        in the sense of \cite[\S5B]{InvBook}: By \cite[Pr.~4.2]{KnusVilla01},\footnote{
            Pr.~4.2 of \cite{KnusVilla01} is phrased in the case $A$ is a division ring, but the proof
            works for any central simple $K$-algebra.
        }
        $(P,[\beta])$ induces a quadratic pair $(\tau,f)$ on the central simple $K$-algebra $E:=\End_A(P)$,
        and one can show  that $O(P,[\beta])$ is canonically isomorphic  to
        \[O(\tau,f):=\{x\in E\suchthat \text{$x^\tau x=1$ and $f(x^{-1}ax)=f(a)$ $\forall\, a\in E$}\}\ .\]
        The map $\Delta:O(P,[\beta])\to\Z/2\Z$
        can now be defined via the Dickson invariant of similitudes of quadratic pairs  introduced in \cite[\S12C]{InvBook}.
        To check the equivalence of this definition with the one previously given,
        it is enough to show that
        conjugation and $e$-transfer do not affect the quadratic pair $(\tau,f)$ constructed
        in \cite[Pr.~4.2]{KnusVilla01}, and that both definitions
        are compatible with extending to a splitting field. Provided that, it is enough
        to check that the definitions coincide when $A=K$, and this follows from \cite[Ex.~12.31]{InvBook}.
        We leave all the details (which seem to be missing in the literature) to the reader.
    \end{remark}

\subsection{Generation by Reflections}

    We now use the Dickson invariant to describe the group $O'(P,[\beta])$.
    Throughout, $(A,\sigma,u,\Lambda)$ is a semiperfect (and at one point semilocal) ring.
    We use the same setting as in \ref{subsection:general-setting} and set some further notation.

\medskip

    Let $(P,[\beta])$ be a quadratic space. Then any $\psi\in O(P,[\beta])$
    gives rise to $\psi_i\in O(P_i,[\beta_i])$ given by $\psi_i(\pi_i x)=\pi_i(\psi_i x)$.
    Observe that if $\psi$ is an $e$-reflection of $(P,[\beta])$ (with $e$ an idempotent of $A$), then
    $\psi_i$ is a $\pi_ie$-reflection of $(P_i,[\beta_i])$. In particular,
    if $\psi$ is a reflection, then so is $\psi_i$.
    Conversely, when the image of $e$ in $\quo{A}$
    lives in $A_i$, every $\pi_ie$-reflection of $(P_i,[\beta_i])$ is induced from an $e$-reflection.
    This fact, which
    easily follows from Lemma~\ref{LM:ef-inv-mod-Jac}, will be used repeatedly henceforth.

    When $(A_i,\sigma_i,u_i,\Lambda_i)$ split-orthogonal and $(P,[\beta])$ is unimodular,
    we define
    \[
    \Delta_i=\Delta_{i,[\beta]}:O(P,[\beta])\to \Z/2\Z
    \]
    by $\Delta_i(\psi)=\Delta_{[\beta_i]}(\psi_i)$ (Lemma~\ref{LM:onto-is-preserved}(ii) implies
    that $(P_i,[\beta_i])$ is unimodular). It is clear
    that $\Delta_i$ is a group homomorphism.

    We shall call $(A_i,\sigma_i,u_i,\Lambda_i)$ \emph{odd} (resp.\ \emph{even}) \emph{split-orthogonal}
    if $(A_i,\sigma_i,u_i,\Lambda_i)$ is split-orthogonal and $n_i$ is odd (resp.\ even; recall that $A_i=\nMat{D_i}{n_i}$).

    Finally, let $\veps_{i}^{(j)}$ denote the matrix unit $e_{jj}$ in $A_i$. Then $\{\veps_{i}^{(j)}\}_{i,j}$ is a complete
    system of orthogonal idempotents in $\quo{A}$, hence it can be lifted to a complete
    system of orthogonal idempotents $\{e_i^{(j)}\}_{i,j}$ in $A$.
    We choose this lifting to have $e_i=e_i^{(1)}$ (where $e_i$ is defined as in~\ref{subsection:general-setting}),
    and set $f_i=\sum_je_i^{(j)}$ (so $\quo{f}_i=1_{A_i}$).

    \begin{lem}\label{LM:reflection-existence-I}
        Let $(P_i,[\beta_i])$ be a unimodular quadratic space over $(A_i,\sigma_i,u_i,\Lambda_i)$.
        \begin{enumerate}
            \item[(i)] If $(A_i,\sigma_i,u_i,\Lambda_i)$ is not split-orthogonal, then $\id_P$ is an $e_i$-reflection
            of $(P,[\beta])$.
            \item[(ii)] If $P_i\neq 0$, then $(P_i,[\beta_i])$ admits an $\veps_i$-reflection.
        \end{enumerate}
    \end{lem}

    \begin{proof}
        (i) If $(A_i,\sigma_i,u_i,\Lambda_i)$ is not split-orthogonal, then so is $(A_{(i)},\sigma_{(i)},u_{(i)},\Lambda_{(i)})$
        (Proposition~\ref{PR:orth-not-affected-by-conj}).
        We claim that $\Lambda_{(i)}\cap\units{A_{(i)}}\neq \emptyset$.
        Indeed, if $A_{(i)}\cong E\times E^\op$ with $E$ a division ring,
        then $(1_E,0_E^\op)-(1_E,0_E^\op)^\sigma_{(i)} u_{(i)}\in\Lambda_{(i)}\cap\units{A_{(i)}}$. If $A_{(i)}$ is a division ring
        and  $\Lambda_{(i)}\cap\units{A_{(i)}}\neq \emptyset$, then $\Lambda_{(i)}=0$.
        This implies, $a^\sigma u_{(i)}=a$ for all $a\in A_{(i)}$, and by taking $a=1$, we get $u_{(i)}=1$ and $\sigma_{(i)}=\id_{A^{(i)}}$.
        But this means $(A_{(i)},\sigma_{(i)},u_{(i)},\Lambda_{(i)})$ is split-orthogonal, a contradiction.

        Now, choose $a\in e_i^\sigma \Lambda e_i$ such that $\pi_i a\in \Lambda_{(i)}\cap\units{A_{(i)}}$. Then $a$ is
        $(e_i^\sigma,e_i)$-invertible (Lemma~\ref{LM:ef-inv-mod-Jac}), hence
        $s_{0,e_i,a}$ is an $e_i$-reflection, and it is easy to see that $s_{0,e_i,a}=\id_P$.

        (ii) If $(A_i,\sigma_i,u_i,\Lambda_i)$ is not split-orthogonal,
        take the projection to $P_i$ of the $e_i$-reflection constructed in (i).
        Otherwise, $A_{(i)}$ is a field and $\Lambda_{(i)}=0$, so
        we need to find $y\in P_{(i)}:=P_i\veps_i$ with $\hat{\beta}_{(i)}(y)\neq 0$.
        Indeed,  $[\beta]$ is unimodular, hence $[\beta_{i}]$ is unimodular
        (Lemma~\ref{LM:onto-is-preserved}(ii)). Since $P_{i}\neq 0$,
        this means $[\beta_{i}]\neq 0$, and hence $[\beta_{(i)}]\neq 0$.
        As $\Lambda_{(i)}=0$, there must be $y\in P_{(i)}$ with $\hat{\beta}_{(i)}(y)\neq 0$,
        as required.
    \end{proof}

    \begin{prp}\label{PR:reflection-existence-II}
        Let $(P,[\beta])$ be a unimodular quadratic space. Then the following conditions
        are equivalent:
        \begin{enumerate}
            \item[(a)] $(P,[\beta])$ admits a reflection.
            \item[(b)] For all $i$, if $(A_i,\sigma_i,u_i,\Lambda_i)$ is odd split-orthogonal, then $P_i\neq 0$.
        \end{enumerate}
        In this case, there exist $f_i$-reflections for all $i$.
    \end{prp}

    \begin{proof}
        Observe that $(P,[\beta])$ has a reflection $\iff$
        $(\quo{P},[\quo{\beta}])$ has a reflection $\iff$
        there exists $x\in \quo{P}$ such
        that $\hat{\quo{\beta}}(x)\cap \units{{\quo{A}}}\neq \emptyset$ $\iff$
        for all $i$, there exists $x\in P_i$ such that $\hat{\beta}_i(x)\cap \units{A_i}\neq \emptyset$
        $\iff$ $(P_i,[\beta_i])$ admits a reflection $\iff$ $(P,[\beta])$ has an $f_i$-reflection for all $i$.

        Assume (b) holds. By Lemma~\ref{LM:reflection-existence-I}, if $(A_i,\sigma_i,u_i,\Lambda_i)$
        is not split-orthogonal or $P_i\neq 0$, then $(P_i,[\beta_i])$ admits an $\veps_i$-reflection $s$.
        By Lemma~\ref{LM:reflections-manipulations}(i), $s$ is also an $\veps_i^{(j)}$-reflection
        for all $j$, hence by Lemma~\ref{LM:reflections-manipulations}(ii), $s^{n_i}$ is a reflection of $(P_i,[\beta_i])$
        (because $1_{A_i}=\sum_{j=1}^{n_i}\veps^{(j)}_i$). If $(A_i,\sigma_i,u_i,\Lambda_i)$ is even split-orthogonal,
        then $D_i$ is a field, $n_i$ is even, $\sigma_i$ is the matrix transpose involution, and $u= 1$.
        It is then easy to see that $\Lambda_i$ contains a unit $a\in \units{A_i}$, so $s_{0,1,a}$ is a reflection
        of $(P_i,[\beta_i])$.

        Assume (b) is false. Then there is $i$ such that $(A_i,\sigma_i,u_i,\Lambda_i)$
        is odd split-orthogonal and $P_i= 0$. Thus, $(P_i,[\beta_i])$ has a reflection if and only if
        $\Lambda_i$ contains units. Since $\Lambda_i$ consists of $n_i\times n_i$ alternating matrices and $n_i$ is odd,
        $\Lambda_i$ does not contain any units and hence $(P_i,[\beta_i])$ has no reflections.
    \end{proof}

    \begin{remark}
        Proposition~\ref{PR:reflection-existence-II} can also be proved using the more general result \cite[Lm.~3.6]{Reiter75}.
    \end{remark}

    \begin{lem}\label{LM:prod-of-two-refl}
        Let $(P,[\beta])$ be a unimodular quadratic space such that
        $P_k\neq 0$ whenever $(A_k,\sigma_k,u_k,\Lambda_k)$ is odd split-orthogonal.
        Then the product of two $e_i$-reflections of $(P,[\beta])$ equals a product of two reflections.
    \end{lem}

    \begin{proof}
        By Proposition~\ref{PR:reflection-existence-II}, there exist
        an
        $f_k$-reflection $s_k$ for all $1\leq k\leq \ell$.
        Let $t$ and $t'$ be $e_i$-reflections. By Lemma~\ref{LM:reflections-manipulations}(i),
        $t$ and $t^{-1}$ are also $e_i^{(j)}$-reflections for all $1\leq j\leq n_i$.
        Thus, by Lemma~\ref{LM:reflections-manipulations}(ii),
        $\psi=t^{n_i}s_1\dots \hat{s}_i\dots s_k$ and $\vphi=s_k^{-1}\dots\hat{s}_i^{-1}\dots s_1^{-1}(t^{-1})^{n_i-1}t'$
        are reflections (here, $\hat{~}$ means omission). But $tt'=\psi\vphi$, so we are done.
    \end{proof}

    Recall that $O'(P,[\beta])$ denotes the subgroup of $O(P,[\beta])$ generated by reflections.

    \begin{thm}\label{TH:gen-by-reflections}
        Let $(P,[\beta])$ be a unimodular quadratic space, and assume that
        \begin{enumerate}
            \item[(1)] if $(A_i,\sigma_i,u_i,\Lambda_i)$ is  split-orthogonal, then $D_i\ncong\F_2$, and
            \item[(2)] if $D_i\cong\F_2\times\F_2$, then $P_i\ncong \veps_i A_i$.
        \end{enumerate}
        Let $\calI$ be the set of indices $1\leq i\leq \ell$ such
        that $P_i\neq 0$ and $(A_i,\sigma_i,u_i,\Lambda_i)$ is split-orthogonal,
        and let $\xi:=(n_i+2\Z)_{i\in \calI}\in (\Z/2\Z)^{\calI}$.
        Denote by $\Delta_{\calI}$ the group homomorphism
        \[
        \psi\mapsto (\Delta_i(\psi))_{i\in \calI}~:~O(P,[\beta])\to (\Z/2\Z)^{\calI}\ .
        \]
        Then:
        \begin{enumerate}
            \item[(i)] $\Delta_{\calI}$ is onto (even without assuming (1) and (2)).
            \item[(ii)] If $P_i\neq 0$ whenever $(A_i,\sigma_i,u_i,\Lambda_i)$ is odd split-orthogonal,
            then $O'(P,[\beta])=\Delta_{\calI}^{-1}(\{0,\xi\})$.
            \item[(ii)] If there is $i$ such that $P_i=0$ and $(A_i,\sigma_i,u_i,\Lambda)$
            is odd split-orthogonal, then $O'(P,[\beta])=1$.
        \end{enumerate}
    \end{thm}

    \begin{proof}
        (i) Let $j\in \calI$. By Proposition~\ref{PR:Dickson-inv-of-reflections}, $\Delta_\calI$
        of an $e_j$-reflection is $(\delta_{ij})_{i\in \calI}$ (where $\delta_{ij}=1$ if $i=j$ and $0$ otherwise),
        and
        by Lemma~\ref{LM:reflection-existence-I}(ii), $e_j$-reflections exist for all $i\in\calI$, so $\Delta_{\calI}$
        is onto.

        (ii) By Proposition~\ref{PR:reflection-existence-II}, $(P,[\beta])$ admits a reflection $s$,
        and by Proposition~\ref{PR:Dickson-inv-of-reflections}, $\Delta_{\calI}(s)=\xi$. This
        also implies that $O'(P,[\beta])\subseteq \Delta_\calI^{-1}(\{0,\xi\})$.
        Now, to prove the other inclusion, it  enough to show that $\ker\Delta_\calI\subseteq O'(P,[\beta])$.
        Let $\psi\in\ker \Delta_{\calI}$.
        By Corollary~\ref{CR:main-III}, $\psi$ is a product of quasi-reflections.
        Moreover, from the proof of Theorem~\ref{TH:Witt-I} it follows that $\psi$ can written
        as a product of $e_1$-reflections followed by a product of $e_2$-reflection etc.,
        and if $P_i=0$, then the product includes no $e_i$-reflections.
        For $1\leq i\leq \ell$ with $P_i\neq 0$, let $m_i$ denote the number of $e_i$-reflection
        used to express $\psi$.
        We claim that $m_i$ can be taken
        to be even. By Lemma~\ref{LM:prod-of-two-refl}, this will imply $\psi\in O'(P,[\beta])$.
        Indeed, if $(A_i,\sigma_i,u_i,\Lambda_i)$  is  split-orthogonal (and $P_i\neq 0$), then by
        Proposition~\ref{PR:Dickson-inv-of-reflections}, $m_i+2\Z=\Delta_i(\psi)=0$, so $m_i$ is even.
        When $(A_i,\sigma_i,u_i,\Lambda_i)$ is not split-orthogonal, we can freely increase $m_i$
        by inserting $\id_P$, which is an $e_i$-reflection by
        Lemma~\ref{LM:reflection-existence-I}(i), into the product, so again, $m_i$ can be made even.

        (iii) This follows from Proposition~\ref{PR:reflection-existence-II}.
    \end{proof}

    \begin{cor}
        Let $n$ (resp.~$m$) denote the number of $i$-s such
        that $(A_i,\sigma_i,u_i,\Lambda_i)$ is odd (resp.~even) split-orthogonal. Then
        for any unimodular quadratic space $(P,[\beta])$ admitting a reflection
        (see Proposition~\ref{PR:reflection-existence-II}), we have
        \[
        [O(P,[\beta]):O'(P,[\beta])]=2^{m+\max\{n-1,0\}}\ .
        \]
    \end{cor}

    We believe that  Theorem~\ref{TH:gen-by-reflections}(ii) should also be true
    when $A$ is semilocal and $P$ is a progenerator (i.e.\ $A$ is a summand of $P^n$ for some $A$,
    or equivalently $P_i\neq 0$ for all $i$).
    Indeed, we have the following.

    \begin{thm}
        Provided $P$ is free, part (ii)
        of Theorem~\ref{TH:gen-by-reflections} holds under the milder
        assumption that $A$ is semilocal.
    \end{thm}

    \begin{proof}
        By Proposition~\ref{PR:reflection-existence-II}, $(\quo{P},[\quo{\beta}])$ admits a reflection,
        which can lifted to a reflection $s$ of $(P,[\beta])$. That reflection satisfies $\Delta_{\calI}(s)=\xi$
        by Proposition~\ref{PR:Dickson-inv-of-reflections}. This proposition also
        implies that $O'(P,[\beta])\subseteq\Delta_{\calI}^{-1}(\{0,\xi\})$,
        so again, it is left to prove that $\ker \Delta_{\calI}\subseteq O'(P,[\beta])$.
        Let $\psi\in\ker\Delta_{\calI}$, and let $\quo{\psi}$ be the isometry
        it induces on $(\quo{P},[\quo{\beta}])$ (namely, $\quo{\psi}(\quo{x})=\quo{\psi x}$).
        Arguing as in the proof of Theorem~\ref{TH:gen-by-reflections},
        we see that $\quo{\psi}$ is   a product of reflections of $(\quo{P},[\quo{\beta}])$.
        These reflections can be lifted to reflections of $(P,[\beta])$, and their product
        is an isometry $\psi'\in O(P,[\beta])$ such that $\quo{\psi}=\quo{\psi'}$.
        Replacing $\psi$ with $\psi'^{-1}\psi$, we may assume $\quo{\psi}=\id_{\quo{P}}$,
        or rather, $\psi x-x\in P\Jac(A)$ for all $x\in P$.
        Now, it is shown in the proof of \cite[Th.~6.2]{Reiter75} that such $\psi$ is a product
        of reflections (here we need $P$ to be free), so we are done.
        (Notice that  \cite[Th.~6.2]{Reiter75} assumes that $[\beta_i]\neq 0$ whenever
        $(A_i,\sigma_i,u_i,\Lambda_i)$ is  split-orthogonal, but the argument
        that we have counted on only needs $[\beta_i]\neq 0$ when
        $(A_i,\sigma_i,u_i,\Lambda_i)$ is \emph{odd} split-orthogonal. Also note that $[\beta_i]\neq 0$
        $\iff$ $P_i\neq 0$ because $(P_i,[\beta_i])$ is unimodular.)
    \end{proof}

    \begin{example}
        Part (i) of Theorem~\ref{TH:gen-by-reflections} may fail when $A$ is only assumed to be semilocal.
        For example, take $A$ to be a non-local semilocal commutative domain with $2\in\units{A}$,  set $\sigma=\id_A$,
        $u=1$, $\Lambda=0$, and define $\beta:A\times A\to A$ by $\beta(x,y)=xy$.
        It is easy  to check that $O(A,[\beta])=\{\pm \id_A\}$. However, $|\calI|>1$ because $\quo{A}$ is not a field,
        and hence $\Delta_\calI$ cannot be onto.
    \end{example}

\bibliographystyle{plain}
\bibliography{MyBib}

\def\Dbar{\leavevmode\lower.6ex\hbox to 0pt{\hskip-.23ex \accent"16\hss}D}
  \def\Dbar{\leavevmode\lower.6ex\hbox to 0pt{\hskip-.23ex \accent"16\hss}D}
  \def\Dbar{\leavevmode\lower.6ex\hbox to 0pt{\hskip-.23ex \accent"16\hss}D}
  \def\Dbar{\leavevmode\lower.6ex\hbox to 0pt{\hskip-.23ex \accent"16\hss}D}
\begin{thebibliography}{10}

\bibitem{Azu51}
Gor{\^o} Azumaya.
\newblock On maximally central algebras.
\newblock {\em Nagoya Math. J.}, 2:119--150, 1951.

\bibitem{Bak69}
Anthony Bak.
\newblock On modules with quadratic forms.
\newblock In {\em Algebraic {K}-{T}heory and its {G}eometric {A}pplications
  ({C}onf., {H}ull, 1969)}, pages 55--66. Springer, Berlin, 1969.

\bibitem{Bass73AlgebraicKThyIII}
Hyman Bass.
\newblock Unitary algebraic {$K$}-theory.
\newblock In {\em Algebraic {K}-theory, {III}: {H}ermitian {K}-theory and
  geometric applications ({P}roc. {C}onf., {B}attelle {M}emorial {I}nst.,
  {S}eattle, {W}ash., 1972)}, pages 57--265. Lecture Notes in Math., Vol. 343.
  Springer, Berlin, 1973.

\bibitem{Ba74}
Hyman Bass.
\newblock Clifford algebras and spinor norms over a commutative ring.
\newblock {\em Amer. J. Math.}, 96:156--206, 1974.

\bibitem{BayerFain96}
Eva Bayer-Fluckiger and Laura Fainsilber.
\newblock Non-unimodular {H}ermitian forms.
\newblock {\em Invent. Math.}, 123(2):233--240, 1996.

\bibitem{BayFiMol13}
Eva Bayer-Fluckiger, Uriya~A. First, and Daniel~A. Moldovan.
\newblock Hermitian categories, extension of scalars and systems of
  sesquilinear forms.
\newblock {\em Pacific J. Math.}, 270(1):1--26, 2014.

\bibitem{BayerMold12}
Eva Bayer-Fluckiger and Daniel~Arnold Moldovan.
\newblock Sesquilinear forms over rings with involution.
\newblock {\em J. Pure Appl. Algebra}, 218(3):417--423, 2014.

\bibitem{Fi13A}
Uriya~A. First.
\newblock General bilinear forms.
\newblock {\em Israel J. Math.}, 205(1):145--183, 2015.

\bibitem{Fi13B}
Uriya~A. First.
\newblock Rings that are {M}orita equivalent to their opposites.
\newblock {\em J. Algebra}, 430:26--61, 2015.

\bibitem{Grove02ClassicalGroups}
Larry~C. Grove.
\newblock {\em Classical groups and geometric algebra}, volume~39 of {\em
  Graduate Studies in Mathematics}.
\newblock American Mathematical Society, Providence, RI, 2002.

\bibitem{Keller88}
Bernhard Keller.
\newblock A remark on quadratic spaces over noncommutative semilocal rings.
\newblock {\em Math. Z.}, 198(1):63--71, 1988.

\bibitem{Kne69}
Manfred Knebusch.
\newblock Isometrien \"uber semilokalen {R}ingen.
\newblock {\em Math. Z.}, 108:255--268, 1969.

\bibitem{Knes69GaloisCohomology}
M.~Kneser.
\newblock {\em Lectures on {G}alois cohomology of classical groups}.
\newblock Tata Institute of Fundamental Research, Bombay, 1969.
\newblock With an appendix by T. A. Springer, Notes by P. Jothilingam, Tata
  Institute of Fundamental Research Lectures on Mathematics, No. 47.

\bibitem{Knes70}
Martin Kneser.
\newblock Witts {S}atz \"uber quadratische {F}ormen und die {E}rzeugung
  orthogonaler {G}ruppen durch {S}piegelungen.
\newblock {\em Math.-Phys. Semesterber.}, 17:33--45, 1970.

\bibitem{Kn91}
Max-Albert Knus.
\newblock {\em Quadratic and {H}ermitian forms over rings}, volume 294 of {\em
  Grundlehren der Mathematischen Wissenschaften [Fundamental Principles of
  Mathematical Sciences]}.
\newblock Springer-Verlag, Berlin, 1991.
\newblock With a foreword by I. Bertuccioni.

\bibitem{InvBook}
Max-Albert Knus, Alexander Merkurjev, Markus Rost, and Jean-Pierre Tignol.
\newblock {\em The book of involutions}, volume~44 of {\em American
  Mathematical Society Colloquium Publications}.
\newblock American Mathematical Society, Providence, RI, 1998.
\newblock With a preface in French by J. Tits.

\bibitem{KnusVilla01}
Max-Albert Knus and Oliver Villa.
\newblock Quadratic quaternion forms, involutions and triality.
\newblock In {\em Proceedings of the {C}onference on {Q}uadratic {F}orms and
  {R}elated {T}opics ({B}aton {R}ouge, {LA}, 2001)}, number Extra Vol., pages
  201--218 (electronic), 2001.

\bibitem{La99}
T.~Y. Lam.
\newblock {\em Lectures on modules and rings}, volume 189 of {\em Graduate
  Texts in Mathematics}.
\newblock Springer-Verlag, New York, 1999.

\bibitem{QuSchSch79}
H.-G. Quebbemann, W.~Scharlau, and M.~Schulte.
\newblock Quadratic and {H}ermitian forms in additive and abelian categories.
\newblock {\em J. Algebra}, 59(2):264--289, 1979.

\bibitem{Reiter75}
H.~Reiter.
\newblock Witt's theorem for noncommutative semilocal rings.
\newblock {\em J. Algebra}, 35:483--499, 1975.

\bibitem{Ro88}
Louis~H. Rowen.
\newblock {\em Ring theory. {V}ol. {I}}, volume 127 of {\em Pure and Applied
  Mathematics}.
\newblock Academic Press Inc., Boston, MA, 1988.

\bibitem{Roy68}
Amit Roy.
\newblock Cancellation of quadratic form over commutative rings.
\newblock {\em J. Algebra}, 10:286--298, 1968.

\bibitem{SchQuadraticAndHermitianForms}
Winfried Scharlau.
\newblock {\em Quadratic and {H}ermitian forms}, volume 270 of {\em Grundlehren
  der Mathematischen Wissenschaften [Fundamental Principles of Mathematical
  Sciences]}.
\newblock Springer-Verlag, Berlin, 1985.

\bibitem{Taylor92GeometryOfClassicalGroups}
Donald~E. Taylor.
\newblock {\em The geometry of the classical groups}, volume~9 of {\em Sigma
  Series in Pure Mathematics}.
\newblock Heldermann Verlag, Berlin, 1992.

\bibitem{Vamos90}
Peter V{\'a}mos.
\newblock Decomposition problems for modules over valuation domains.
\newblock {\em J. London Math. Soc. (2)}, 41(1):10--26, 1990.

\bibitem{Wall70}
C.~T.~C. Wall.
\newblock On the axiomatic foundations of the theory of {H}ermitian forms.
\newblock {\em Proc. Cambridge Philos. Soc.}, 67:243--250, 1970.

\bibitem{Warfield80}
R.~B. Warfield, Jr.
\newblock Cancellation of modules and groups and stable range of endomorphism
  rings.
\newblock {\em Pacific J. Math.}, 91(2):457--485, 1980.

\end{thebibliography}

\end{document}